\documentclass{amsart}
\usepackage{latexsym,amssymb,amsmath,mathabx,stmaryrd}
\usepackage[pdftex]{graphicx}
\usepackage[mathscr]{euscript}
\usepackage{rotating}
\usepackage{color}
\usepackage[color,emtex,matrix,arrow]{xy}
\usepackage{xypic}
\usepackage{comment}

\DeclareMathOperator{\lcm}{lcm}

\newtheorem{thm}{\sc Theorem}[section]
\newtheorem{cor}[thm]{\sc Corollary}
\newtheorem{ex}[thm]{\sc Example}

\newtheorem{prop}[thm]{\sc Proposition}
\newtheorem{prop-defn}[thm]{\sc Proposition-Definition}
\newtheorem{defn}[thm]{\sc Definition}
\newtheorem{rem}[thm]{\sc Remark}
\numberwithin{equation}{section}

\newcommand{\NN}{\mathbb{N}}

\newcommand{\RR}{\mathbb{R}}

\newcommand{\ZZ}{\mathbb{Z}}

\newcommand{\nbf}{\mathbf{n}}

\begin{document}

\title{Normal submonoids and congruences on a monoid}
\author{Josep Elgueta}
\address{Departament de Matem\`atiques, Universitat Polit\`ecnica de Catalunya}
\email{josep.elgueta@upc.edu}
\thanks{The author is supported by the project PID2019-103849GB-I00 of MCIN/AEI/10.13039/501100011033.}
\subjclass[2010]{20M10}
\keywords{monoids; congruences}

\maketitle

\begin{abstract}
A notion of {\em normal submonoid} of a monoid $M$ is introduced that generalizes the normal subgroups of a group. When ordered by inclusion, the set $\mathsf{NorSub}(M)$ of normal submonoids of $M$ is a complete lattice. Joins are explicitly described, and the lattice is computed for the finite full transformation monoids $T_n$, $n\geq 1$. It is also shown that $\mathsf{NorSub}(M)$ is modular for a specific family of commutative monoids, including all Krull monoids, and that, as a join semilattice, embeds isomorphically onto a join subsemilattice of the lattice $\mathsf{Cong}(M)$ of congruences on $M$. This leads to a new strategy for computing $\mathsf{Cong}(M)$ consisting of computing $\mathsf{NorSub}(M)$, and the lattices of the so called unital congruences on the quotients of $M$ modulo its normal submonoids. This provides a new perspective on Malcev computation of the congruences on $T_n$.
\end{abstract}


\section{Introduction}
It is well known that congruences on a group $G$, and normal subgroups of $G$ are essentially the same. Thus every congruence on $G$ is completely determined by the equivalence class of the identity element $1\in G$, this equivalence class is always a normal subgroup of $G$, and every normal subgroup $N$ of $G$ is the identity element equivalence class of a congruence on $G$ (namely, the smallest congruence on $G$ containing $\{1\}\times N$). In other words, there is a canonical bijection
\[
\Psi_G:\mathsf{Cong}(G)\to\mathsf{NorSub}(G)
\]
mapping every congruence $R$ to the equivalence class $[1]_R$ whose inverse
\[
\Phi_G:\mathsf{NorSub}(G)\to\mathsf{Cong}(G)
\]
maps each normal subgroup $N$ to the unique congruence $R_N$ such that $[1]_{R_N}=N$, and whose equivalence classes are the (left or right) cosets of $G$ modulo $N$. In fact, both $\Psi_G$ and $\Phi_G$ are lattice isomorphisms when the domain and the codomain are ordered by inclusion. Indeed, both $\Psi_G$ and $\Phi_G$ are order-preserving, and every order-preserving bijection between lattices with an order-preserving inverse automatically preserves joins and meets because the lattice operations can be defined in terms of the ordering.

As it is also well known, things are not so simple for arbitrary monoids.  In fact, the equivalence class of the identity element modulo a congruence no longer determines the congruence. Thus if $M$ is a monoid, and $I$ is any ideal of $M$ (a nonempty subset $I$ such that $MIM\subseteq I$) the equivalence relation $R_I$ on $M$ given by
\[
(a,b)\in R_I\ \Longleftrightarrow\ \mbox{$a=b$ or $a,b\in I$}
\]
is a congruence, called the {\em Rees congruence} of $I$. It follows that for every {\em proper} ideal $I$ of $M$ the congruence $R_I$ is what we shall call a {\em unital congruence}, i.e. a congruence $R$ such that $[1]_{R}=\{1\}$. In general, these are not the only unital congruences on a monoid. Otherwise, finding all congruences on a given monoid will be much easier, as it will become clear in the sequel (cf. Theorem~\ref{conjunt_congruencies_M} below). For instance, if $M$ is the additive monoid $\NN_+=(\NN,+,0)$ it is easy to check that
\begin{equation}\label{congruencies_N}
R_{m,n}=\Delta_\NN\cup\{(i,j)\in\NN\times\NN\,:\,i,j\geq m,\,i\equiv j\, (\mathrm{mod.}\,n)\}\subseteq\NN\times\NN
\end{equation}
is a congruence on $\NN_+$ for each $m\geq 0$, and $n\geq 1$. In particular, all congruences $R_{m,n}$ with $m>0$ are unital, but only the congruences $R_{m,1}$ for every $m>0$ are Rees congruences.

The main purpose of this paper is to prove a weaker version of the above lattice isomorphism $\mathsf{NorSub}(G)\cong\mathsf{Cong}(G)$ that holds true for arbitrary monoids, and to discuss a strategy for computing the lattice of congruences of a monoid which follows from it. The key notion is that of a {\em normal submonoid} of a monoid. It will play the role of the normal subgroups of a group. 

The paper is organized as follows. In Section~2 a notion of {\em normal submonoid} of a monoid $M$ is introduced that we argue it plays the role of the normal subgroups of a group. Thus we shall see that the preimage $f^{-1}(1)$ for every monoid homomorphism $f:M\to N$, and the equivalence class of the identity element $1\in M$ modulo any congruence on $M$ are both normal submonoids of $M$ in our sense (Propositions~\ref{nucli_es_submonoide_normal} and \ref{lema_classe_neutre}, respectively), and that every normal submonoid of $M$ is the preimage $f^{-1}(1)$ of some monoid homomorphism $f:M\to N$ (Corollary~\ref{submonoide_normal=nucli}). As expected, the set $\mathsf{NorSub}(M)$ of normal submonoids of a monoid $M$, when ordered by inclusion, is a lattice with meets and joins respectively given by the intersection, and the normal closure of the union. The normal closure of an arbitrary subset $A$ of the monoid is explicitly described in terms of $A$ (Proposition~\ref{descripcio_clausura_normal}). As an example, the lattice $\mathsf{NorSub}(T_n)$ of the finite full transformation monoids $T_n$ is computed for each $n\geq 1$. In particular, it is shown that $T_n$ is ``normally simple'', i.e. such that its only proper normal submonoids are the normal subgroups of its group of units (Proposition~\ref{Tn_normalment_simple}). The modularity of the lattice $\mathsf{NorSub}(M)$ is shown for a particular type of monoids, including all free commutative monoids and, more generally, all Krull monoids (Theorem~\ref{monoides_modulars}, and Example~\ref{exemple_monoides_modulars}). In Section~3 it is shown a weakening of the lattice isomorphism $\mathsf{NorSub}(G)\cong\mathsf{Cong}(G)$ that holds for arbitrary monoids. More specifically, it is shown that for every monoid $M$ the map $\Phi_M:\mathsf{NorSub}(M)\to\mathsf{Cong}(M)$ sending each normal submonoid $S$ of $M$ to the smallest congruence $R_S$ on $M$ containing $\{1\}\times S$ is join-preserving, and a one-to-one map with left inverse given by the map $\Psi_M:\mathsf{Cong}(M)\to\mathsf{NorSub}(M)$ sending each congruence to the equivalence class of the identity element (Theorem~\ref{Phi_M_injectiva}). This leads to a possible strategy to compute the lattice of congruences on a monoid $M$ which essentially consists of finding the lattice of normal submonoids of $M$, and the lattices of unital congruences on the so called `normal quotients' of $M$, i.e. the quotients of $M$ modulo the congruences in the image of $\Phi_M$ (Theorem~\ref{conjunt_congruencies_M}). This way the problem of computing the congruences on the monoids reduces to just computing the unital ones. This strategy provides a new perspective on the classical theorem by Malcev describing all congruences on the finite full transformation monoids \cite{Malcev1952}, and it is used to show that these monoids are ``congruentially simple'', i.e. such that all congruences are completely determined by the equivalence class of the identity element except when this class reduces to $\{1\}$. Although this result is an easy consequence of Malcev theorem, we prove it without using this theorem.

\bigskip
\noindent
{\em Note}. One week before finishing this work, Martins-Ferreira and Sobral put on the web a preprint \cite{MartinsFerreira-Sobral2022} where it appears the same notion of normal submonoid used here. In Remark~1 of this preprint, the authors mention a paper by Facchini and Rodaro \cite{Facchini-Rodaro2017} where this notion seems to be introduced for the first time. The present work has been done completely independently from both papers.

\section{Normal submonoids of a monoid}

In what follows, $M=(M,\cdot,1)$ stands for a monoid, and $U(M)$ for its group of {\em units} (i.e. the elements $x\in M$ having a two-sided inverse $x^{-1}\in M$). For every subset $T\subseteq M$ we shall denote by $\langle T\rangle$ the smallest submonoid containing $T$. Thus $\langle T\rangle$ consists of all finite products of elements in $T$, including the empty product (equal to $1$ by convention). In case $M$ is commutative, we shall use additive notation. In particular, the identity element will be denoted by $0$ instead of $1$.

\subsection{Groupal, invariant, and normal submonoids} By a {\em subgroup} of a monoid $M$ it is usually meant a subsemigroup (i.e. a non-empty subset closed with respect to $\cdot$) which is a group with the induced binary operation. Notice that the identity element of a generic subgroup of $M$ need not coincide with the identity $1$ of $M$ so that a generic subgroup of $M$ need not be a submonoid. For instance, for every idempotent $e\neq 1$ of $M$ the set $\{e\}$ is a subgroup of $M$ but not a submonoid. 

Together with this notion, however, there is also the following one, of more interest to us, which also reduces to the notion of subgroup when $M$ is a group.

\begin{defn}
A {\em groupal submonoid} of $M$ is a submonoid $S\subseteq M$ closed under inverses, i.e. such that $x^{-1}\in S$ for each $x\in S\cap U(M)$.
\end{defn}

\noindent
Clearly, every subgroup $H$ of $U(M)$ is a groupal submonoid, but the converse is false in general. Thus if $U(M)=\{1\}$ (for instance, if $M$ is the additive monoid $\NN_+=(\NN,+,0)$ of natural numbers), every non-trivial submonoid of $M$ is automatically a groupal submonoid not included in $U(M)$. On the other hand, in a generic monoid there are submonoids which are not groupal. For instance,  for every $n\times n$ invertible matrix $A$ whose inverse is not a power of itself the subset $\{A^k,\,k\geq 0\}$ is a submonoid of the multiplicative monoid $M_n(k)$ of all $n\times n$ matrices but not a groupal submonoid.

For every subset $T\subseteq M$ we shall denote by $\langle T\rangle'$ the smallest groupal submonoid containing $T$ (equivalently, the intersection of all groupal submonoids containing $T$). In general, it is different from the submonoid $\langle T\rangle$ generated by $T$ because it consists of all finite products of elements of $T$ and/or their inverses (when invertible). For instance, if $M=M_2(\RR)$ then
\[
\left\langle \left(\begin{array}{cc} 1&0\\0&2\end{array}\right)\right\rangle=\left\{\left(\begin{array}{cc} 1&0\\0&2\end{array}\right)^k,\,k\geq 0\right\}
\]
while
\[
\left\langle \left(\begin{array}{cc} 1&0\\0&2\end{array}\right)\right\rangle'=\left\{\left(\begin{array}{cc} 1&0\\0&2\end{array}\right)^k,\,k\in\ZZ\right\}.
\]
Actually, we are interested in the analog for monoids of the normal subgroups of a group. We want them to be a particular type of submonoids such that the equivalence class of the identity element modulo any congruence is of this type. At first sight, we might be tempted to take as analog the groupal submonoids $S$ which are invariant under conjugation by arbitrary invertible elements in $M$ (i.e. such that $xSx^{-1}\subseteq S$ for every $x\in U(M)$) or, more generally, such that $xSy\subseteq S$ for each $x,y\in M$ such that $xy\in S$. For every congruence, the equivalence class of the identity element indeed satisfies this condition. However, this condition turns out to be too weak if we also want that every `normal submonoid' be the identity element equivalence class modulo some congruence on $M$, as it is the case for the normal subgroups of a group. In fact, if $S$ is invariant in this sense it may happen that the smallest congruence $R$ on $M$ for which $S\subseteq [1]_R$ is a congruence such that $S\neq[1]_R$. If so, there will be no congruence with $S$ as the equivalence class of $1$. The fact that this can indeed happen will become clear in the sequel. As argued below, the right notion of invariance turns out to be as follows.
 
\begin{defn}
A subset $A\subseteq M$ is called {\em invariant} if for each $x,y\in M$ the following conditions are equivalent:
\begin{itemize}
\item[(i)] $xAy\subseteq A$,
\item[(ii)] $xy\in A$,
\item[(iii)] $(xAy)\cap A\neq\emptyset$.
\end{itemize}
\end{defn}
It is worth introducing the (two-sided) `stability set' of any subset $A$, and its generalized conjugates, and to rewrite the invariance condition of a {\em submonoid} in terms of it.

\begin{defn}
The (two-sided) {\em stability set} (relative to $M$) of a subset $A\subseteq M$ is the set 
\[
X_A:=\{(x,y)\in M\times M\,:\, (xAy)\cap A\neq\emptyset\}.
\]
The subsets $xAy$ for each $(x,y)\in X_A$ will be called the {\em generalized conjugates} of $A$ (in $M$).
\end{defn}
\noindent
It readily follows from the definition that $X_A\subseteq X_{B}$ if $A\subseteq B$. Moreover, for a {\em submonoid} $S$ the stability set $X_S$ clearly contains $S\times S$ as well as the set $\{(x,x^{-1}),\,x\in U(M)\}$. This is why the subsets $xAy$ for each $(x,y)\in X_A$ are called the generalized conjugates of $A$: when $A$ is a submonoid $S$ they include the usual conjugates $xSx^{-1}$ of $S$ for each $x\in U(M)$. In general, however, $X_S$ contains many pairs other than these. For instance
\[
X_{\{1\}}=\{(x,y)\in M\times M\,:\,xy=1\},
\]
and the equation $xy=1$ may have solutions which are neither in $S\times S$ nor of the form $(x,x^{-1})$ for some $x\in U(M)$. Finally, let us remark that when $M$ is {\em commutative} $X_S$ is a submonoid of $M\times M$ (containing the submonoid $S\times S$). 

In terms of the stability set, and the generalized conjugates the invariance condition on a submonoid $S$ can be restated as follows.

\begin{prop}
A submonoid $S\subseteq M$ is invariant if and only if 
\begin{equation}\label{condicio_invariant}
xSy\subseteq S\qquad \forall\,(x,y)\in X_S,
\end{equation}
i.e. iff all its generalized conjugates are subsets of itself.
\end{prop}

\begin{proof}
If $A$ is a submonoid $S$, not just a subset, condition (i) in the definition of invariant subset implies (ii) because $xy=xey\in A$, and (ii) implies (iii) because $xy=xey\in (xAy)\cap A$. Hence conditions (i) to (iii) are equivalent if and only if (iii) implies (i).
\end{proof}

Notice that every submonoid satisfying (\ref{condicio_invariant}) is invariant in the previous weaker sense, but not conversely. For instance, a submonoid of a monoid with a trivial group of units is automatically invariant in the weak sense, but not necessarily in our stronger sense. An example is provided by the monoid $\NN_+$ (see Proposition~\ref{submonoides_invariants_de_N} below). This leads us to the following analog for arbitrary monoids of the normal subgroups of a group.

\begin{defn}

A {\em normal submonoid} of a monoid $M$ is a submonoid which is both groupal and invariant.
\end{defn}

\begin{rem}{\rm
Usually, a submonoid $S$ of a monoid $M$ is called normal when $xS=Sx$ for each $x\in M$.~\footnote{\,For instance, this is the notion of `normal submonoid` which appears in the groupprops webpage; see https://groupprops.subwiki.org/wiki/Normal\_submonoid.} In fact, every subgroup $H$ of $U(M)$ normal in this sense is a normal submonoid in our sense, as the reader may easily check. However, even restricting to subgroups $H$ of $U(M)$, this condition is not equivalent to the invariance condition~(\ref{condicio_invariant}). For instance, $U(T_n)=S_n$ is a normal submonoid of $T_n$ in our sense (see Proposition~\ref{U(M)_normal} below), but it is not true that $f\,S_n=S_n\,f$ for every transformation $f\in T_n$ if $n>1$ (just take $n=2$, and $f$ any constant map). The same problem persists with the (non-equivalent) condition $xSx^{-1}\subseteq S$ for each invertible element $x\in U(M)$.
}
\end{rem}

If $S$ is a normal submonoid of $M$ we shall write $S\lhd M$. Clearly, if $S\lhd M$, and $N$ is a submonoid of $M$ containing $S$ then $S\lhd N$. But a normal submonoid of a normal submonoid of $M$ need not be a normal submonoid of $M$. In fact, this is already so when $M$ is a group.

Next results provide examples of normal submonoids. The first two show that our normal submonoids can be thought of as analogs for the generic monoids of the normal subgroups of a group. Additional arguments in favor of this viewpoint are discussed in the sequel. The third one highlights the difference between the submonoids
\begin{equation}\label{langle n rangle}
\langle n\rangle:=\{kn,\,k\geq 0\},\qquad n\geq 0
\end{equation}
of $\NN_+$, and all other submonoids of $\NN_+$.

\begin{prop}
The normal submonoids of a group $G$ are the normal subgroups of $G$.
\end{prop}

\begin{proof}
Let be $S$ a normal submonoid of $G$, and let us supose that for some $x\in G$ and $s\in S$ we have $xsx^{-1}\notin S$. Using the equivalence between (ii) and (iii) with $y=sx^{-1}$ we conclude that
\[
(xS(sx^{-1}))\cap S=\emptyset.
\]
However, $e=xx^{-1}=xs^{-1}(sx^{-1})\in(xS(sx^{-1}))\cap S$, contradicting our conclusion. Therefore $xSx^{-1}\subseteq S$ for every $x\in G$. Conversely, let $S$ be a normal subgroup of $G$. We have to see that $S$ satisfies (\ref{condicio_invariant}). Indeed, let us suppose that $(xSy)\cap S\neq\emptyset$ for some $x,y\in G$. This means that there exist $s_1,s_2\in S$ such that $xs_1y=s_2$. Then for each $s\in S$ we have $xsy=\tilde{s}s_2$ with $\tilde{s}=xss^{-1}x^{-1}$. Now, $\tilde{s}\in S$ because $S$ is a normal subgroup of $G$. Hence $xsy\in S$, and $xSy\subseteq S$. 
\end{proof}

\begin{prop}\label{nucli_es_submonoide_normal}
Let $f:M\to N$ be a monoid homomorphism. Then $f^{-1}(1)$ is a normal submonoid of $M$.
\end{prop}

\begin{proof}
Clearly, $f^{-1}(1)$ is always a groupal submonoid of $M$. Let us assume that for some $x,y\in M$ there exists $z_1,z_2\in f^{-1}(1)$ such that $xz_1y=z_2$. Then $1_N=f(z_2)=f(xz_1y)=f(x)f(y)$ and hence, for every $z\in f^{-1}(1)$ we have $f(xzy)=f(x)f(y)=1_N$, i.e. $x\,f^{-1}(1)\,y\subseteq f^{-1}(1)$. Thus $f^{-1}(1)$ is also invariant.
\end{proof}

\begin{prop}\label{submonoides_invariants_de_N}
The normal submonoids of $\NN_+$ are the submonoids (\ref{langle n rangle}).
\end{prop}

\begin{proof}
The submonoid $\langle n\rangle$ is clearly groupal. Moreover, if $(k,l)\in X_{\langle n\rangle}$ there exists $q,q'\geq 0$ such that $k+nq+l=nq'$. Hence $q'\geq q$, and $k+l\in\langle n\rangle$ so that $k+\langle n\rangle+l\subseteq\langle n\rangle$, i.e. $\langle n\rangle$ is invariant. Let us now prove that these are the only normal submonoids of $\NN_+$. Since all submonoids of $\NN_+$ are groupal, we have to see that these are the only invariant submonoids. To prove this, let be $S$ any submonoid of $\NN_+$, and let be $n:=\min(S\setminus\{0\})$, so that $\langle n\rangle\subseteq S$. If $S\neq\langle n\rangle$ let be $n':=\min(S\setminus\langle n\rangle)$. Then we have $n'=nq'+r$ for some $0<r<n$ and $q'\geq 1$ ($q'\neq 0$ because $n$ is the least nonzero element of $S$). It follows that $n'\in(0+S+r)\cap S$, i.e. $(0,r)\in X_S$. However, $0+S+r\nsubset S$ because $r\in 0+S+r$ while $r\notin S$.
\end{proof}

\begin{rem}{\rm 
It follows from the previous result that in a {\em commutative} monoid not every groupal submonoid is necessarily invariant. For instance, the subset $S_n:=\{k,\,k\geq n\}\cup\{0\}\subseteq\NN$ is a groupal submonoid of $\NN_+$ for each $n\geq 2$ but not a normal submonoid. Hence the set (in fact, lattice) of normal submonoids of a generic commutative monoid does not coincide with the set of all submonoids, as it occurs when the monoid is a group. }
\end{rem}

Let us finish this paragraph with two more examples of normal submonoids. The second one provides another example of a commutative monoid having groupal submonoids which are not invariant.

\begin{ex}\label{N_max}{\rm
Let $\NN_{\max}$ be the set $\NN$ of natural numbers equipped with the product given by
\[
mn:=\max(m,n).
\]
$\NN_{\max}$ is a commutative monoid with $0$ as identity element. We claim that $\{0,1,\ldots,n\}$ is a normal submonoid of $\NN_{\max}$ for each $n\geq 0$. It is clearly closed under products, but also under inverses because $0$ is the only invertible element. Moreover, the reader may easily check that
\[
X_{\{0,1,\ldots,n\}}=\{0,1,\ldots,n\}^2.
\]
Since the product of elements $\leq n$ is $\leq n$ it follows that $\{0,1,\ldots,n\}$ is also invariant.
}
\end{ex}

\begin{ex}\label{monoide_biciclic}{\rm
Let $\NN\boxtimes\NN$ be the set $\NN\times\NN$ equipped with the product given by
\[
(m,n)(p,q)=(m+\max(n,p)-n,q+\max(n,p)-p).
\]
$\NN\boxtimes\NN$ is a commutative monoid, called the {\em bicyclic monoid}, with $(0,0)$ as identity element (it is isomorphic to the monoid defined by the presentation $\langle a,b\,|\,ab=1\rangle$; see \cite{Howie1995}, p.~31-32). Then the diagonal $\Delta=\{(n,n),\,n\geq 0\}$ is a normal submonoid. Indeed, it is clearly closed under products. Moreover, the reader may easily check that $(m,n)(p,q)=(0,0)$ if and only if $m=q=0$ and $n=p$, while $(n,0)(0,n)=(n,n)$. Hence $(0,0)$ is the only invertible element, from which it follows that $\Delta$ is also closed under inverses and hence, a groupal submonoid. To prove it is invariant notice that
\[
(k,l)(n,n)(r,s)=(k+\max(l,n,r)-l,s+\max(l,n,r)-r)
\]
so that $(k,l)(n,n)(r,s)\in\Delta$ if and only if $k-l=s-r$, independently of $n$. Hence
\[
X_{\Delta}=\{((k,l),(r,r+k-l)),\,k,l,r\geq 0\},
\]
and $\Delta$ indeed contains all its generalized conjugates. Notice that, as a monoid, $\Delta$ is isomorphic to the monoid in Example~\ref{N_max} and hence, $S_n:=\{(0,0),(1,1),\ldots,(n,n)\}$ is a normal submonoid of $\Delta$ for each $n\geq 0$. However, $S_n$ for $n\geq 2$ is not invariant as a groupal submonoid of $\NN\boxtimes\NN$. Thus $((n-1,1),(2,n))\in X_{S_n}$ because $(n-1,1)(2,n)=(n,n)$ but for $n>2$ we have $(n-1,1)(n,n)(2,n)=(2n-2,2n-2)\notin S_n$.
}
\end{ex}

\subsection{Normal closure}\label{clausura_normal}
As for groups, given any subset $A$ of a generic monoid $M$, by the {\em normal closure} of $A$ (in $M$) we shall mean the smallest normal submonoid of $M$ containing $A$. It will be denoted by $\mathrm{ncl}_M(A)$, or just $\mathrm{ncl}(A)$ if no confusion arises, and by $\mathrm{ncl}_M(x_1,\ldots,x_n)$ in the finite subset case $A=\{x_1,\ldots,x_n\}$.

Next result implies that $\mathrm{ncl}_M(A)$ is given as usually, i.e. by the intersection of all normal submonoids of $M$ containing $A$.

\begin{prop}
The intersection of normal submonoids of $M$ is a normal submonoid of $M$.
\end{prop}

\begin{proof}
Let $\{S_\lambda\}_{\lambda\in\Lambda}$ be a family of normal submonoids of $M$, and let be
\[
S:=\bigcap_{\lambda\in\Lambda}S_\lambda.
\]
Clearly, $S$ is a groupal submonoid. Moreover, since $S\subseteq S_\lambda$ we have $X_{S}\subseteq X_{S_\lambda}$ for each $\lambda\in\Lambda$. Thus for each $(x,y)\in X_{S}$ we have
\[
xSy\subseteq xS_\lambda y\subseteq S_\lambda
\]
because each $S_\lambda$ is invariant. Hence $xSy\subseteq S$, i.e. $S$ is also invariant.
\end{proof}

\begin{cor}
For every subset $A\subseteq M$ we have
\[
\mathrm{ncl}_M(A)=\bigcap_{S\in\mathsf{NorSub}_A(M)}S
\]
with $\mathsf{NorSub}_A(M)$ the set of all normal submonoids of $M$ containing $A$.
\end{cor}

\noindent
Furthermore, as any closure operator, the normal closure operator is idempotent, preserves inclusions, and it is such that for any subsets $A,B$ of $M$
\begin{equation}\label{clausura_reunio_subconjunts}
\mathrm{ncl}_M(A\cup B)=\mathrm{ncl}_M(\mathrm{ncl}_M(A)\cup\mathrm{ncl}_M(B)).
\end{equation}
When $M$ is a group $G$, the normal closure of $A$ is the subgroup of $G$ generated by the set of all conjugacy classes of elements in $A$. For a generic monoid, things are a little bit more sophisticated, and $\mathrm{ncl}_M(A)$ is built from $A$ as follows. Let be $A_n$ for $n\geq 0$ the groupal submonoids of $M$ recursively defined by
\begin{align}
A_0&:=\langle A\rangle', \label{def_An-1}
\\
A_n&:=\langle\bigcup_{\tiny{(x,y)\in X_{A_{n-1}}}}xA_{n-1}y\ \ \rangle',\quad n\geq 1 \label{def_An-2}
\end{align}
with $\langle \rangle'$ the {\em groupal} submonoid generated by $A$, and $X_A$ the stability set of $A$ for any subset $A\subseteq M$. Notice that
\[
A_0\subseteq A_1\subseteq \cdots\subseteq A_n\subseteq \cdots
\]
because $(1,1)\in X_{A}$ for each $A$. Consequently, we also have
\[
X_{A_0}\subseteq X_{A_1}\subseteq\cdots\subseteq X_{A_{n}}\subseteq\cdots
\]

\begin{prop}\label{descripcio_clausura_normal}
For every subset $A\subseteq M$ the normal closure of $A$ in $M$ is
\[
\mathrm{ncl}_M(A)=\bigcup_{n\geq 0}A_n.
\]
\end{prop}

\begin{proof}
Every normal submonoid of $M$ containing $A$ must contain $\mathrm{ncl}_M(A)$. Hence it is enough to see that $\mathrm{ncl}_M(A)$ is a normal submonoid. Let be $z,z'\in \mathrm{ncl}_M(A)$. Then $z\in A_n$ and $z'\in A_{n'}$ for some $n,n'\geq 0$. Since $A_n,A_{n'}\subseteq A_{\max\{n,n'\}}$ it follows that $z,z'\in A_{\max\{n,n'\}}$ and hence, $zz'\in A_{\max\{n,n'\}}\subseteq \mathrm{ncl}_M(A)$ because $A_{\max\{n,n'\}}$ is a submonoid. Moreover, if $z\in U(M)$ we also have $z^{-1}\in A_{\max\{n,n'\}}$ because it is a {\em groupal} submonoid. This proves that $\mathrm{ncl}_M(A)$ is a groupal submonoid of $M$. To prove that it is invariant, let us first observe that
\begin{equation}\label{X_A-barret}
X_{\mathrm{ncl}_M(A)}=\bigcup_{n\geq 0}X_{A_n}.
\end{equation}
Indeed, since $A_n\subseteq \mathrm{ncl}_M(A)$ we have  $X_{A_n}\subseteq X_{\mathrm{ncl}_M(A)}$ for each $n\geq 0$ and hence, $\bigcup_{n\geq 0}X_{A_n}\subseteq X_{\mathrm{ncl}_M(A)}$. Conversely, let $(x,y)\in X_{\mathrm{ncl}_M(A)}$. This means that there exist $t^*_1,t^*_2\in \mathrm{ncl}_M(A)$ such that $x t^*_1y=t^*_2$ and hence, some $n_1,n_2\geq 0$ such that $t^*_1\in A_{n_1}$, $t^*_2\in A_{n_2}$, and $xt^*_1y=t^*_2$. Since $A_{n_1},A_{n_2}\subseteq A_{\max\{n_1,n_2\}}$ it follows that $(x,y)\in A_{\max\{n_1,n_2\}}\subseteq\bigcup_{n\geq 0}X_{A_n}$. Let be now $(x,y)\in X_{\mathrm{ncl}_M(A)}$, and let us prove that $x\mathrm{ncl}_M(A)y\subseteq \mathrm{ncl}_M(A)$. Indeed, it follows from (\ref{X_A-barret}) that $(x,y)\in X_{A_{n_0}}$ for some $n_0\geq 0$ and hence, $(x,y)\in X_{A_n}$ for each $n\geq n_0$. Then for any element $t^*\in \mathrm{ncl}_M(A)$ we have $t^*\in A_{m_0}$ for some $m_0\geq 0$ and hence, $t^*\in A_m$ for each $m\geq m_0$. If $m_0\leq n_0$ we have $t^*\in A_{n_0}$ and $(x,y)\in X_{A_{n_0}}$, from which it follows that $xt^*y\in A_{n_0+1}\subseteq \mathrm{ncl}_M(A)$. Similarly, if $m_0>n_0$ we have $t^*\in A_{m_0}$ and $(x,y)\in X_{A_{m_0}}$, from which it follows that $xt^*y\in A_{m_0+1}\subseteq \mathrm{ncl}_M(A)$.  
\end{proof}

Notice that for a commutative monoid $M=(M,+,0)$ the generating set of $A_n$ for $n\geq 1$ can also be described by
\[
\bigcup_{(x,y)\in X_{A_{n-1}}}(x+A_{n-1}+y)=\bigcup_{z\in LX_{A_{n-1}}}(z+A_{n-1}),
\]
where $LX_A$ for any subset $A\subseteq M$ stands for the (left) {\em one-sided stability set} of $A$, i.e.
\[
LX_A:=\{z\in M\,|\,(z+A)\cap A\neq\emptyset\}.
\]
Thus if $z\in LX_A$ then $(z,0)\in X_A$ and conversely, if $(x,y)\in X_A$ then $x+y\in LX_A$. Hence the set in either side is indeed a subset of the set in the other side.  

\begin{ex}\label{clausures_normals_a_N}{\rm
For every $n_1,\ldots,n_k\geq 0$ with $k\geq 1$ the normal closure of $\{n_1,\ldots,n_k\}$ is
\[
\mathrm{ncl}_{\NN_+}(n_1,\ldots,n_k)=\left\{
\begin{array}{ll}
\langle n_1\rangle, & \mbox{if $k=1$,}
\\
\langle\gcd(n_1,\ldots,n_k)\rangle, & \mbox{if $k\geq 2$.}
\end{array}\right.
\]
Let us first prove the case $k=1$. We have to see that
\begin{equation}\label{clausura_normal_n}
\mathrm{ncl}_{\NN_+}(n)=\langle n\rangle
\end{equation}
for each $n\geq 0$. Since $U(M)=\{0\}$ we have $\{n\}_0=\langle n\rangle'=\langle n\rangle$. Now $LX_{\langle n\rangle}$ consists of the positive integers $p\in\NN$ such that $p+\langle n\rangle$ contains some multiple of $n$ and consequently, of those $p\geq 0$ such that $p\in\langle n\rangle$. Hence for each $p\in LX_{\langle n\rangle}$ we have $p+\langle n\rangle\subseteq\langle n\rangle$, and $\{n\}_1=\{n\}_0$, from which the claim readily follows. Let us now prove the cases $k\geq 2$ by induction on $k$. Let be $k=2$. If $n_1=n_2$ the result follows from (\ref{clausura_normal_n}). If $n_1\neq n_2$ we have
\[
\{n_1,n_2\}_0=\langle n_1,n_2\rangle=\{kn_1+ln_2,\,k,l\geq 0\}.
\]
Now, assuming that $n_1<n_2$ we have $n_2=qn_1+r_1$ for some $q\geq 1$ and $0<r_1<m$. Hence $r_1\in LX_{\langle n_1,n_2\rangle}$. Thus we have
\[
\{n_1,n_2\}_1=\langle(r_1+\langle n_1,n_2\rangle)\cup\cdots\rangle\supseteq\langle r_1,n_1\rangle.
\]
By the same argument applied to the subset $\{r_1,n_1\}$ it follows that
\[
\{n_1,n_2\}_2=\langle(r_2+\langle r_1,n_1\rangle)\cup\cdots\rangle\supseteq\langle r_2,r_1\rangle,
\]
where $r_2$ is the remainder of the euclidean division of $n_1$ by $r_1$. In particular, $r_2\in\{n_1,n_2\}_2$. After a finite number $p\geq 1$ of iterations we find that 
\[
\gcd(n_1,n_2)\in\{n_1,n_2\}_p\subseteq\mathrm{ncl}_{\NN_+}(n_1,n_2)
\]
and hence, $\langle\gcd(n_1,n_2)\rangle\subseteq\mathrm{ncl}_{\NN_+}(n_1,n_2)$. The equality follows because $\langle\gcd(n_1,n_2)\rangle$ is already a normal submonoid containing $n_1,n_2$. Let be now $k>2$, and let us assume that $\mathrm{ncl}_{\NN_+}(n_1,\ldots,n_{k-1})=\langle\gcd(n_1,\ldots,n_{k-1})\rangle$ for every $n_1,\ldots,n_{k-1}\geq 0$. Then we have
\begin{align*}
\mathrm{ncl}_{\NN_+}(n_1,\ldots,n_k)&=\mathrm{ncl}_{\NN_+}(\{n_1,\ldots,n_{k-1}\}\cup\{n_k\})
\\
&=\mathrm{ncl}_{\NN_+}(\mathrm{ncl}_{\NN_+}(n_1,\ldots,n_{k-1})\cup\mathrm{ncl}_{\NN_+}(n_k))
\\
&=\mathrm{ncl}_{\NN_+}(\langle\gcd(n_1,\ldots,n_{k-1})\rangle\cup\langle n_k\rangle)
\\
&=\mathrm{ncl}_{\NN_+}(\mathrm{ncl}_{\NN_+}(\gcd(n_1,\ldots,n_{k-1}))\cup\mathrm{ncl}_{\NN_+}(n_k))
\\
&=\mathrm{ncl}_{\NN_+}(\gcd(n_1,\ldots,n_{k-1}),n_k)
\\
&=\langle\gcd(\gcd(n_1,\ldots,n_{k-1}),n_k)\rangle
\\
&=\langle\gcd(n_1,\ldots,n_k)\rangle.
\end{align*}
}
\end{ex}

\begin{ex}\label{exemple_clausura_normal_a_Nmax}{\rm
For each $n\geq 0$ the normal closure of $\{n\}$ in $\NN_{\max}$ (see Example~\ref{N_max}) is given by
\begin{equation}\label{clausura_normal_a_N_max}
\mathrm{ncl}_{\NN_{\max}}(n)=\{0,1,\ldots,n\}.
\end{equation}
Indeed, since $U(\NN_{\max})=\{0\}$ we have 
\[
\{n\}_0=\langle n\rangle'=\langle n\rangle=\{0,n\}.
\]
Now it is easy to check that $LX_{\{0,n\}}=\{0,1,\ldots,n\}$ and hence,
\[
\{n\}_1=\{0,1,\ldots,n\}.
\]
Then (\ref{clausura_normal_a_N_max}) follows because this is already a normal submonoid. More generally, an easy induction on $k\geq 1$ shows that
\[
\mathrm{ncl}_{\NN_{\max}}(n_1,\ldots,n_k)=\{0,1,\ldots,\max(n_1,\ldots,n_k)\}.
\]
The details are left to the reader.
}
\end{ex}

\subsection{Normal, and normally simple monoids}
As a (groupal) submonoid of $M$, the group of units $U(M)$ is invariant if (and only if) for every pair $(x,y)\in M\times M$ for which there exists some $u\in U(M)$ such that $xuy\in U(M)$ we have $xu'y\in U(M)$ for each $u'\in U(M)$. There seems to be no reason by which this should be true for a generic monoid $M$. This suggests introducing the following definition.

\begin{defn}
A monoid $M$ is called {\em normal} when its group of units $U(M)$ is a normal submonoid.
\end{defn} 

\noindent
Clearly, every group, finite or not, is a normal monoid, as it is every monoid whose group of units is trivial, such as any free monoid, and the bicyclic monoid $\NN\boxtimes\NN$ of Example~\ref{monoide_biciclic}. Many more examples are provided by the next result.

\begin{prop}\label{U(M)_normal}
Let be $M$ a monoid. Then $M$ is normal in any of the following two cases:
\begin{itemize}
\item[(i)] $M$ is finite;
\item[(ii)] $M$ is commutative.
\end{itemize}
\end{prop} 

\begin{proof}
(i) Let $M$ be a finite monoid which is not a group, i.e. such that $M\setminus U(M)\neq\emptyset$. Then the stability set of $U(M)$ is
\begin{equation}\label{X_U(M)}
X_{U(M)}=U(M)\times U(M),
\end{equation}
from which the statement readily follows. The inclusion $U(M)\times U(M)\subseteq X_{U(M)}$ is obvious. As to the reverse inclusion, it is ultimately a consequence of the so-called {\em stability property} of finite monoids, according to which for finite monoids $M$, and for every $x,y\in M$ the two following conditions are equivalent:
\begin{itemize}
\item[(1)] $MxM=MxyM$ (resp. $MxM=MyxM$),
\item[(2)] $xM=xyM$ (resp. $Mx=Myx$).
\end{itemize}
Using this equivalence, it is not difficult to see that
\[
U(M)=\{x\in M\,:\,MxM=M\},
\]
and that its complement $M\setminus U(M)$, nonempty by hypothesis, is a (two-sided) ideal of $M$; see for ex. \cite{Steinberg2016}, Section~1.3 for the details. Now, if $M\setminus U(M)$ is an ideal, and for some $x,y\in M$ there exists some $u\in U(M)$ such that $xuy\in U(M)$ we necessarily have $x,y\in U(M)$. Hence $X_{U(M)}=U(M)\times U(M)$. 

(ii) If $M$ is commutative, and $xuy\in U(M)$ for some $u\in U(M)$ it follows that $xy\in U(M)$ and hence, $xu'y=xyu'\in U(M)$ for each $u'\in U(M)$.
\end{proof}

\begin{prop}\label{subgrups_normals_U(M)}
Let $M$ be a normal monoid. Then every normal subgroup of $U(M)$ is a normal submonoid of $M$.  
\end{prop} 

\begin{proof}
Every subgroup $H$ of $U(M)$ is trivially a groupal submonoid of $M$. Let us assume that $H$ is a normal subgroup of $U(M)$, and that for some $x,y\in M$ we have $(xHy)\cap H\neq\emptyset$. Since $H\subseteq U(M)$, we also have $(xU(M)y)\cap U(M)\neq\emptyset$ and hence, $xU(M)y\subseteq U(M)$ because $U(M)$ is a normal submonoid of $M$. Then for every $h\in H$ we have $xhy\in U(M)$.   
\end{proof}

In general, a normal monoid $M$ will have normal submonoids other than itself, and the normal subgroups of $U(M)$. For instance, this is so for the monoid $\NN_+$ (cf. Proposition~\ref{submonoides_invariants_de_N}), and for the monoids in Examples~\ref{N_max} and \ref{monoide_biciclic}.

\begin{defn}
A normal monoid $M$ is called {\em normally simple} if its only normal submonoids are $M$, and the normal subgroups of $U(M)$.
\end{defn}

Clearly, every group is a normally simple monoid. An important family of normally simple noncommutative monoids which are not groups is the family of full transformation monoids $T_n:=End(\nbf)$ for each $n\geq 1$, where $\nbf:=\{1,\ldots,n\}$. 

\begin{prop}\label{Tn_normalment_simple}
$T_n$ is normally simple for each $n\geq 1$.
\end{prop}

\begin{proof}
$T_n$ is a normal monoid because it is finite. To prove that it is normally simple it is enough to see that for each $f\in T_n$ the normal closure of $\{f\}$ in $T_n$ is given by
\begin{equation}\label{ncl(f)}
\mathrm{ncl}_{T_n}(f)=\left\{
\begin{array}{ll}
\mathrm{ncl}_{S_n}(f), & \mbox{if $f\in S_n$,}
\\
T_n, & \mbox{if $f\notin S_n$}
\end{array}\right.
\end{equation}
where $\mathrm{ncl}_{S_n}(f)$ stands for the normal closure of $\{f\}$ in the symmetric group $S_n$. Indeed, let this be true, and let be $S$ a normal submonoid of $T_n$. If $S\setminus S_n\neq\emptyset$, and $f\in S\cap S_n$ it follows from (\ref{ncl(f)}) that $T_n=\mathrm{ncl}_{T_n}(f)\subseteq S$ and hence, $S=T_n$. Otherwise, we have $S\subseteq S_n$, and being a normal submonoid of $T_n$, $S$ is also a normal submonoid of $S_n$ and hence, a normal subgroup of $S_n$.

Let us prove (\ref{ncl(f)}). We need to compute the sequence of groupal submonoids $\{f\}_0,\{f\}_1,\ldots$. Let us first consider the case where $f$ is a permutation. By definition we have
\[
\{f\}_0=\langle f\rangle'=\{f^p,\,p\in\ZZ\}.
\]
Then the stability set of $\{f\}_0$ consists of all pairs $(g,h)\in T_n\times T_n$ such that $gf^ph=f^q$ for some $p,q\in\ZZ$. Since $f$ is a bijection, $h$ must be injective, and $g$ surjective. Being endomorphisms of $\nbf$, it follows that $g,h\in S_n$. Hence we can take as $g$ any permutation in $S_n$, and $h=f^qg^{-1}f^{-p}$, i.e.
\[
X_{\{f\}_0}=\{(\tau,f^q\tau^{-1}f^{-p}),\,\tau\in S_n,\,p,q\in\ZZ\}.
\]
In particular, $(\tau,\tau^{-1})\in X_{\{f\}_0}$ for each $\tau\in S_n$. Thus
\[
\mathrm{ncl}_{T_n}(f)\supseteq\{f\}_1=\langle\bigcup_{(g,h)\in X_{\{f\}_0}}g\{f\}_0h\ \rangle'\supseteq\langle \bigcup_{\tau\in S_n}\tau\{f\}_0\tau^{-1}\,\rangle'\supseteq\mathrm{ncl}_{S_n}(f).
\]
But $\mathrm{ncl}_{S_n}(f)$ is already a normal submonoid of $T_n$ containing $f$ because of Proposition~\ref{subgrups_normals_U(M)}. Therefore $\mathrm{ncl}_{T_n}(f)=\mathrm{ncl}_{S_n}(f)$ when $f$ is invertible.

Let now be $f$ a non-invertible endomorphism. To prove that $\mathrm{ncl}(f)=T_n$ we proceed by induction on the rank $k<n$ of $f$. If $k=1$ there exists $i\in\nbf$ such that $f=c_i$, the constant function mapping all of $\nbf$ to $i$. Now,
\[
\{c_i\}_0=\{id_{\nbf},c_i\},
\]
and for every $h\in T_n$ we have
$c_ih=c_i$, i.e. $(id_{\nbf},h)\in X_{\{c_i\}_0}$. It follows that for every $h\in T_n$ we have
\[
h\in id_{\nbf}\{c_i\}_0h\subseteq\{c_i\}_1\subseteq\mathrm{ncl}_{T_n}(f),
\]
i.e. $\mathrm{ncl}_{T_n}(f)=T_n$. Let us now assume that $ncl_{T_n}(\phi)=T_n$ for every endomorphism $\phi\in T_n$ of rank $l\in\{1,\ldots,k-1\}$, with $1<k<n$. Since $f$ is not invertible we have
\[
\{f\}_0=\{f^p,\,p\geq 0\}.
\]
Let be $\mathrm{Im}f=\{i_1<i_2<\cdots<i_k\}\subset\nbf$, and $\{A_1,\ldots,A_k\}$ the partition of $\nbf$ such that $f^{-1}(i_j)=A_j$. Thus $f=\varepsilon\eta$, with $\eta:\nbf\to\{i_1,\ldots,i_k\}$ given by $A_j\mapsto i_j$, and $\varepsilon:\{i_1,\ldots,i_k\}\to\nbf$ the canonical inclusion. Then the powers of $f$ can be described in terms of the map $\alpha:\{i_1,\ldots,i_k\}\to\{i_1,\ldots,i_k\}$ defined by
\[
\alpha(i_j)=i_{j'},
\]
where $j'$ is the unique element in $\{1,\ldots,k\}$ such that $i_j\in A_{j'}$. Thus $f^2$ is given by $A_j\overset{f}{\mapsto} i_j\overset{f}{\mapsto} i_{j'}$, i.e. we have $f^2=\varepsilon\alpha\eta$. More generally, for each $p\geq 1$ we have
\[
f^p=\varepsilon\alpha^{p-1}\eta.
\]
The map $\alpha$ can be bijective or not. If it is not bijective, it follows that the powers of $f$ all have rank less than $k$. But if $\{f\}_0$ contains a map $\phi$ of rank $l<k$ it follows by the induction hypothesis that $T_n=\mathrm{ncl}_{T_n}(\phi)$, and hence $\mathrm{ncl}_{T_n}(f)=T_n$ because $\mathrm{ncl}_{T_n}(\phi)\subseteq\mathrm{ncl}_{T_n}(f)$. It remains to prove that the same is true when $\alpha$ is not bijective. In this case, all powers of $f$ are of rank $k$. However, we claim that when $\alpha$ is not bijective $\{f\}_1$ contains maps of rank less than $k$, so that the same argument can be applied. Indeed, since $f$ is not a bijection there is some $j\in\{1,\ldots,k\}$ such that $A_j$ contains more than one element. Without loss of generality, let us assume that $|A_1|>2$. Moreover, since $\alpha$ is a bijection, each $i_j$ belongs to a different $A_{j'}$ and hence, there is some $a_1\in A_1$ such that $a_1\notin\{i_1,\ldots,i_k\}$. Then let us consider the pair $(g,h)\in T_n\times T_n$ given as follows:
\begin{itemize}
\item[(i)] $h$ maps all of $A_1$ to $a_1$, and each $A_j$ for $j\in\{2,\ldots,k\}$ to any $a_j\in A_j$;
\item[(ii)] $g$ acts as the identity on $\{i_1,\ldots,i_k\}$, and collapses $\nbf\setminus\{i_1,\ldots,i_k\}$ to $i_2$ (in particular, we have $g(a_1)=i_2$).
\end{itemize}
Then we have $gfh=f$, i.e. $(g,h)\in X_{\{f\}_0}$ and hence, $gh\in\{f\}_1$. However $gh$ is of rank at most $k-1$ because even if $\{a_2,\ldots,a_k\}=\{i_2,\ldots,i_k\}$ we have $a_1$ which is also mapped to $i_2$.
\end{proof}

\subsection{Lattice of normal submonoids, and modularity}
As the set of normal subgroups of a group, the set $\mathsf{NorSub}(M)$ of normal submonoids of a monoid $M$ is a complete lattice when ordered by inclusion, with meets and joins respectively given by
\begin{align*}
\bigwedge_{i\in I}S_i&=\bigcap_{i\in I}S_i,
\\
\bigvee_{i\in I}S_i&=\mathrm{ncl}_M\left(\bigcup_{i\in I}S_i\right)
\end{align*}
for every family of normal submonoids $\{S_i\}_{i\in I}$. In general, computing this lattice for a given monoid $M$ may be quite difficult, if possible at all, or the resulting lattice may be a complex one. In some cases, however, it is just a (finite or infinite) chain.

\begin{ex}\label{exemple_reticle_submonoides_normals}{\rm
Let us write $[n]:=\{0,1,\ldots,n\}$ for each $n\geq 0$. Then the lattice of normal submonoids of $\NN_{\max}$ is the infinite chain
\[
[0]\subset[1]\subset\cdots\subset[n]\subset\cdots\NN.
\]
Indeed, let $S$ be a normal submonoid of $\NN_{\max}$. If $S$ is a bounded set, and $n$ is its maximum, we have $S=[n]$ because $\mathrm{ncl}_{\NN_{\max}}(n)=[n]$ (cf. Example~\ref{exemple_clausura_normal_a_Nmax}). Otherwise for each $k\geq 0$ there is some element $n\in S$ with $n>k$ because $S$ is not bounded. Therefore $k\in\mathrm{ncl}_{\NN_{\max}}(n)\subseteq S$, and $S=\NN$.  
}
\end{ex}

\begin{ex}\label{reticle_submonoides_normals_Tn}{\rm
It readily follows from Proposition~\ref{Tn_normalment_simple}, together with the simplicity of the alternating groups $A_n$ for each $n\geq 5$ that the lattice of normal submonoids of $T_n$ is given by the finite chain
\begin{align*}
\{1\}\subset S_2\subset T_2,\quad &\mbox{if $n=2$,}
\\
\{1\}\subset K_4\subset A_4\subset S_4\subset &T_4,\quad \mbox{if $n=4$,}
\\
\{1\}\subset A_n\subset S_n\subset T_n,&\quad \mbox{if $n\neq 2,4$.}
\end{align*}
with $K_4=\{id,(12)(34),(13)(24),(14)(32)\}$ the Klein permutation four-group. }
\end{ex}

\begin{defn}
A monoid $M$ is called {\em modular} if the lattice of normal submonoids $\mathsf{NorSub}(M)$ is modular (i.e. if $S_1\vee(S_2\wedge S_3)=(S_1\vee S_2)\wedge S_3$ for every normal submonoids $S_1S_2,S_3$ of $M$ such that $S_1\subseteq S_3$).
\end{defn}

As it is well known, every group is a modular monoid. Thus the question arises whether every monoid is also modular, and in case it is not, determining what monoids, or families of monoids, are modular. The standard proof of the modularity of $\mathsf{NorSub}(G)$ for a group $G$ makes use of the fact that the join of two normal subgroups is nothing but their product (see, for instance, \cite{Jacobson1985-I}, Theorem~8.3). However, as a matter of fact the equality $S\vee S'=SS'$ for every normal submonoids $S,S'$ of a monoid $M$ is not necessary for $\mathsf{NorSub}(M)$ to be modular. Actually, there are modular monoids for which this equality does not hold for each $S,S'$. Even more, $SS'$ need not be a normal submonoid at all when $M$ is not a group.

\begin{ex}\label{modularitat_N+}{\rm
Let be $M=\NN_+$. If $S=\langle 2\rangle$, and $S'=\langle 3\rangle$ the product $SS'$, more properly denoted by $S+S'$ in this case, is $\langle 2\rangle+\langle 3\rangle=\NN\setminus\{1\}$, which is not a normal submonoid of $\NN_+$. Hence $\langle 2\rangle\vee\langle 3\rangle\neq\langle 2\rangle+\langle 3\rangle$. In spite of that, $\NN_+$ is modular. Indeed, let be $S_1,S_2,S$ normal submonoids of $\NN_+$, with $S_1\subseteq S_2$. We have to see that $(S_1\vee S)\wedge S_2=S_1\vee(S\wedge S_2)$. Now, it follows from Proposition~\ref{submonoides_invariants_de_N} that $S_1=\langle nq\rangle$, $S_2=\langle n\rangle$, and $S=\langle m\rangle$ for some $m,n,q\geq 0$. Moreover, from Example~\ref{clausures_normals_a_N} we know that
\begin{align*}
(\langle nq\rangle\vee\langle m\rangle)\wedge\langle n\rangle&=\langle\gcd(nq,m)\rangle\cap\langle n\rangle=\langle\lcm(\gcd(nq,m),n)\rangle
\\
\langle nq\rangle\vee(\langle m\rangle\wedge\langle n\rangle)&=\langle nq\rangle\vee\langle\lcm(m,n)\rangle=\langle\gcd(nq,\lcm(m,n)\rangle.
\end{align*}
Hence we just need to see that $\lcm(\gcd(nq,m),n)=\gcd(nq,\lcm(m,n))$, and this is a consequence of the general fact that $\gcd(a,\lcm(b, c)) =\lcm(\gcd(a, b),\gcd(a, c))$ for each $a,b,c$. }
\end{ex}

As shown by this example, the modularity of a monoid $M$ is a more subtle question than just knowing if the equality $S\vee S'=SS'$ holds for every normal submonoids $S,S'$ of $M$. Next result gives an infinite family of monoids $M$ such that $\mathsf{NorSub}(M)$ is modular, a family that generalizes the monoid in Example~\ref{modularitat_N+}.

\begin{thm}\label{monoides_modulars}
Let be $M=(M,+,0)$ a cancellative commutative monoid such that
\[
\mbox{(*)\ \ \ $\forall x,y\in M,\ \exists z\in M$ such that $x=y+z$ or $y=x+z$,}.
\]
Then $\mathsf{NorSub}(M)$ is isomorphic to the lattice $\mathsf{Sub}(\widehat{M})$ of (normal) subgroups of the Grothendieck group $\widehat{M}$ of $M$. In particular, every cancellative commutative monoid satisfying (*) is modular.
\end{thm}

\begin{proof}
Since $M$ is assumed to be cancellative, $\widehat{M}$ is given (up to isomorphism) by the quotient of the product monoid $M\times M$ modulo the congruence relation
\[
(m,n)\sim(m',n')\ \Leftrightarrow\ m+n'=m'+n.
\]
We shall denote by $[m,n]$ the equivalence class of $(m,n)$ in $\widehat{M}$. Then for every normal submonoid $S\lhd M$ let be $\widehat{S}$ the subgroup of $\widehat{M}$ given by
\[
\widehat{S}:=\{u\in\widehat{M}\,|\,\mbox{$u=[s,s']$ for some $(s,s')\in S\times S$}\}.
\]
The notation is justified by the fact that this subgroup is indeed isomorphic to the Grothendieck group of $S$. This is again a consequence of the cancellative character of $M$, which ensures that $(s,s'),(t,t')\in S\times S$ are equivalent in the Grothendieck group of $S$ if and only if they are equivalent in $\widehat{M}$. Then let be $f:\mathsf{NorSub}(M)\to\mathsf{Sub}(\widehat{M})$ the map defined by $S\mapsto\widehat{S}$.

We claim that $f$ is a lattice isomorphism. It is clearly order-preserving. Moreover, it is injective. Thus let be $S_1,S_2\lhd M$ such that $\widehat{S}_1=\widehat{S}_2$, and let be $x\in S_1$. Then $[x,0]\in\widehat{S}_1$ and hence, we also have $[x,0]\in\widehat{S}_2$. But this means that there exists a pair $(s_2,s'_2)\in S_2\times S_2$ such that $[x,0]=[s_2,s'_2]$, i.e. such that $x+s'_2=s_2$. It follows that $x$ belongs to the left stability set of $S_2$ and consequently, $x=x+0\in S_2$ because $S_2$ is invariant. This proves that $S_1\subseteq S_2$. A similar argument proves that $S_2\subseteq S_1$ and hence, $S_1=S_2$. To prove it is surjective, let be
\[
n(H):=\{x\in M\,:\,[x,0]\in H\}
\]
for every (normal) subgroup $H$ of $\widehat{M}$. We claim that $n(H)$ is a normal submonoid of $M$ such that
\begin{equation}\label{w(n(H))=H}
\widehat{n(H)}=H.
\end{equation}
Indeed, it is a groupal submonoid of $M$ because $H$ is a subgroup of $\widehat{M}$. Moreover, let be $z\in LX_{n(H)}$ so that $z+x\in n(H)$ for some $x\in n(H)$. This means that $[z+x,0],[x,0]\in H$ and hence,
\[
[z,0]=[z+x,x]=[z+x,0]+[0,x]=[z+x,0]-[x,0]
\]
is also in $H$. Therefore $z\in n(H)$ and consequently, $z+n(H)\subseteq n(H)$, i.e. $n(H)$ is invariant. Let us now prove (\ref{w(n(H))=H}). The inclusion $\widehat{n(H)}\subseteq H$ holds even if $M$ does not satisfy (*). Thus if $u\in\widehat{n(H)}$ then $u=[x,x']$ for some $x,x'\in n(H)$ and hence, such that $[x,0],[x',0]\in H$. Then we have
\[
u=[x,0]+[0,x']=[x,0]-[x',0]\in H
\]
because $H$ is a subgroup. It is to prove the reverse inclusion that condition (*) is needed. Thus let be $h\in H$, i.e. $h=[y,y']$ for some $y,y'\in M$. By condition (*) we have either $y=y'+z$ or $y'=y+z$ for some $z\in M$. In the first case we have
\[
h=[y'+z,y']=[z,0]\in H
\]
and hence, $z\in n(H)$ and $h=[z,0]\in\widehat{n(H)}$. Similarly, in the second case we have
\[
h=[y,y+z]=[0,z]=-[z,0]\in H
\]
and hence, also $[z,0]\in H$ because $H$ is a subgroup. Consequently, $z\in n(H)$ and again $h\in\widehat{n(H)}$. This proves that $f$ is a bijective order-preserving map. Since the inverse map $H\mapsto n(H)$ is also order-preserving, it follows that $f$ is a lattice isomorphism. Last assertion follows then from the modularity of the lattice of normal subgroups of a group.
\end{proof}

\begin{ex}\label{exemple_monoides_modulars}{\rm
Every free commutative monoid is modular. Indeed, every free commutative monoid is cancellative and satisfies condition (*). More generally, let us recall that by a saturated submonoid $S$ of a commutative monoid $M$ it is meant a submonoid such that for every $s_1,s_2\in S$ and $x\in M$ such that $s_1=s_2+x$ we have $x\in S$. Then every saturated submonoid of a free commutative monoid is also a cancellative commutative monoid satisfying (*) and hence, modular. Even more generally, every Krull monoid is modular, where by a Krull monoid it is meant a monoid isomorphic to $A\times S$ for some abelian group $A$, and some saturated submonoid $S$ of a free commutative monoid. 
}
\end{ex}

Let us emphasize that condition (*) is needed to prove that $f$ is onto with inverse map given by $H\mapsto n(H)$. For a generic cancellative commutative monoid we have the following weaker version of Theorem~\ref{monoides_modulars}.

\begin{thm}
For every concellative commutative monoid $M$, $\mathsf{NorSub}(M)$ is isomorphic, as join semilattice, to a join subsemilattice of $\mathsf{Sub}(\widehat{M})$.
\end{thm}

\begin{proof}
We have to see that the above injective map $f:\mathsf{NorSub}(M)\to\mathsf{Sub}(\widehat{M})$ is join-preserving. Let be $S_1,S_2\lhd M$. The inclusion $\widehat{S}_1\vee\widehat{S}_2\subseteq\widehat{S_1\vee S}_2$ follows because $f$ is order-preserving, and $\widehat{S}_1\vee\widehat{S}_2$ is the smallest subgroup containing both $\widehat{S}_1$ and $\widehat{S}_2$. To prove the reverse inclusion, let be $x\in\widehat{S_1\vee S}_2$. This means that $x=[t,t']$ for some $t,t'\in S_1\vee S_2$. Since $\widehat{S}_1\vee\widehat{S}_2=\widehat{S}_1+\widehat{S}_2$ we have to see that
\begin{equation}\label{(t,t')}
(t,t')\sim(s_1+s_2,s'_1+s'_2)
\end{equation}
for some pairs $(s_1,s'_1)\in S_1\times S_1$, and $(s_2,s'_2)\in S_2\times S_2$. To prove this, recall that
\[
S_1\vee S_2=\mathrm{ncl}_M(S_1\cup S_2)=\bigcup_{n\geq 0}(S_1\cup S_2)_n,
\]
with $(S_1\cup S_2)_n$ the sequence of groupal submonoids recursively defined by 
\begin{align*}
(S_1\cup S_2)_0&=\langle S_1\cup S_2\rangle',
\\
(S_1\cup S_2)_n&=\langle\bigcup_{z\in LX_{(S_1\cup S_2)_{n-1}}}(z+(S_1\cup S_2)_{n-1})\rangle',\quad n\geq 1
\end{align*}
(see \S~\ref{clausura_normal}). Since the sequence $(S_1\cup S_2)_n$ is a non-decreasing chain $t,t'\in S_1\vee S_2$ implies that $t,t'\in (S_1\cup S_2)_n$ for some $n\geq 0$. Then we prove (\ref{(t,t')}) by induction on $n\geq 0$. If $n=0$ we have $t,t'\in\langle S_1\cup S_2\rangle$ because both $S_1,S_2$ are groupal submonoids. Hence both are finite sums of elements in $S_1$ and/or $S_2$. By the commutativity of $M$, it follows that $(t,t')$ is in fact equal to a pair as in the right hand side of (\ref{(t,t')}). Let us now assume that for some $n\geq 1$ every pair $(u,u')$ with $u,u'\in (S_1\cup S_2)_{n-1}$ is equivalent to a pair as in the right hand side of (\ref{(t,t')}), and let be $t,t'\in(S_1\cup S_2)_n$. This means that there exist $z_1,\ldots,z_{k+l},z'_1,\ldots,z'_{k'+l'}\in LX_{(S_1\cup S_2)_{n-1}}$, and $u_1,\ldots,u_{k+l},u'_1,\ldots,u'_{k'+l'}\in (S_1\cup S_2)_{n-1}$ such that $z_{i}+u_{i},z'_{i'}+u'_{i'}$ are invertible for each $i\in\{k+1,\ldots,k+l\}$ and $i'\in\{k'+1,\ldots,k'+l'\}$, and $(t,t')=(A+B,A'+B')$ with
\begin{align*}
A&=(z_1+u_1)+\cdots+(z_k+u_k),
\\
A'&=(z'_1+u'_1)+\cdots+(z'_{k'}+u'_{k'}),
\\
B&=[-(z_{k+1}+u_{k+1})]+\cdots+[-(z_{k+l}+u_{k+l})],
\\
B'&=[-(z'_{k'+1}+u'_{k'+1})]+\cdots+[-(z'_{k'+l'}+u'_{k'+l'})]
\end{align*}
(here we are using the commutativity of $M$ to write first all `positive terms' in the expressions of both $t,t'$). But 
\begin{align*}
(t,t')&=(A+B,A'+B')
\\
&\sim(A+B,A'+B')+((-B)+(-B'),(-B)+(-B'))
\\
&=(A+(-B'),A'+(-B))
\end{align*}
Moreover, since $M$ is commutative, and both $LX_{(S_1\cup S_2)_{n-1}}$ and $(S_1\cup S_2)_{n-1}$ are submonoids reordering terms we have
\begin{align*}
A+(-B')&=y+v,
\\
A'+(-B)&=y'+v',
\end{align*}
with $y\in LX_{(S_1\cup S_2)_{n-1}}$, $v\in(S_1\cup S_2)_{n-1}$ given by
\begin{align*}
y&=z_1+\cdots+z_k+z'_{k'+1}+\cdots+z'_{k'+l'},
\\
v&=u_1+\cdots+u_k+u'_{k'+1}+\cdots+u'_{k'+l'},
\end{align*} 
and similarly $y',v'$. Now, since $y,y'\in LX_{(S_1\cup S_2)_{n-1}}$ there exists $w,w'\in(S_1\cup S_2)_{n-1}$ such that $y+w,y'+w'\in(S_1\cup S_2)_{n-1}$. Therefore
\[
(t,t')\sim(y+v,y'+v')\sim(y+w+v+w',y'+w'+v'+w)\in (S_1\cup S_2)_{n-1}.
\]
By the induction hypothesis it follows that $(t,t')$ is also equivalent to a pair as in the right hand side of (\ref{(t,t')}).  Hence
\[
x=[t,t']=[s_1,s'_1]+[s_2,s'_2]\in\widehat{S}_1+\widehat{S}_2,
\]
and $\widehat{S}_1\vee\widehat{S}_2\supseteq\widehat{S_1\vee S}_2$.
\end{proof}

\begin{rem}{\rm 
It seems that for a generic cancellative commutative monoid $M$ the lattice $\mathsf{NorSub}(M)$ is not isomorphic through the injection $f$ to a sublattice of $\mathsf{Sub}(\widehat{M})$, a fact which will imply the modularity of $M$ also in this more general case. The problem is that condition (*) seems to be also necessary to prove that $f$ is meet-preserving, and not just onto. Indeed, the inclusion $\widehat{S_1\cap S}_2\subseteq\widehat{S}_1\cap\widehat{S}_2$ is always true because $S_1\cap S_2\subseteq S_1,S_2$ and hence, $\widehat{S_1\cap S}_2$ is a subgroup contained in both $\widehat{S}_1$ and $\widehat{S}_1$. However, in order to prove the reverse inclusion we have to see that everytime we have $u=[s_1,s'_1]=[s_2,s'_2]$ for some pairs $(s_1,s'_1)\in S_1\times S_1$, and $(s_2,s'_2)\in S_2\times S_2$ then there is a pair $(x,y)$ with $x,y\in S_1\cap S_2$ whose equivalence class is $u$. Using hypothesis (*) this is easily shown. For instance, if $s_1=x+s'_1$ or $s'_1=x+s_1$ for some $x\in M$ then $x\in LX_{S_1}$ and hence, $x\in S_1$ because $S_1$ is invariant. In other words, we have either $u=[x,0]$, or $u=[0,x]$ for some $x\in S_1$. Similarly, the invariance of $S_2$ shows that either $u=[y,0]$, or $u=[0,y]$ for some $y\in S_2$. It is now easy to see that in either case we have $x,y\in S_1\cap S_2$ and hence, that $u\in\widehat{S_1\cap S_2}$. For instance, if $u=[x,0]=[y,0]$ we have $x=y$, while $u=[x,0]=[0,y]$ implies that $x+y=0$, i.e. $y$ is the opposite of $x$ and hence, we also have $y\in S_1$ because $S_1$ is groupal. However, it is not clear that without assuming (*) the reverse inclusion is still true.
}
\end{rem}

\section{On the lattice of congruences on a monoid}

As recalled in the introduction, the lattice of congruences on a group $G$ is isomorphic to the lattice of normal subgroups of $G$. The purpose of this section is to prove a weaker version of this result valid for a generic monoid $M$. More precisely, we shall see that $\mathsf{NorSub}(M)$ embeds canonically into the set of congruences on $M$, and that this embedding is an isomorphism of join semilattices between $\mathsf{NorSub}(M)$ and the join subsemilattice of the so called `normal congruences' on $M$. We shall also describe a general procedure to compute the ``blow up'' of $\mathsf{NorSub}(M)$ giving the whole lattice of congruences on $M$. As we shall see, it basically reduces the problem of finding all congruences to being able to compute the unital ones.

In what follows $\mathsf{Cong}(M)$ denotes the set of congruences on a monoid $M$. When ordered by inclusion, it is a complete lattice with meets and joins respectively given by
\begin{align*}
\bigwedge_{i\in I}R_i&=\bigcap_{i\in I}R_i,
\\
\bigvee_{i\in I}R_i&=\left(\bigcup_{i\in I}R_i\right)^\sharp
\end{align*}
for any family of congruences $\{R_i\}_{i\in I}$ on $M$. Here $Y^\sharp$ for every subset $Y\subseteq M\times M$ denotes the smallest congruence on $M$ containing $Y$. Explicitly, it is the equivalence relation on $M$ generated by the subset $\{(xay,xby),\,(a,b)\in Y,\,x,y\in M\}$ (see \cite{Howie1995}, Proposition~1.5.8, or \cite{Cain2020}, Propositions~1.27 and 1.29).

\subsection{Congruence induced by a subset}
\label{congruencia_induida_per_T}
For every subset $A\subseteq M$ let us denote by $\mathsf{Cong}(M,A)$ the subset of $\mathsf{Cong}(M)$ consisting of the congruences $R$ such that $A\subseteq[1]_R$. Notice that both the meet and join in $\mathsf{Cong}(M)$ of a family of congruences in $\mathsf{Cong}(M,A)$ still are in $\mathsf{Cong}(M,A)$. Hence $\mathsf{Cong}(M,A)$ is a complete sublattice of $\mathsf{Cong}(M)$. Let be $R_A$ the least element in $\mathsf{Cong}(M,A)$, i.e.
\[
R_A:=(\{1\}\times A)^\sharp.
\]
It will be called the congruence on $M$ induced by $A$. Explicitly, $R_A$ is given as follows. If $v,w\in M$ we shall say that $w$ is an {\em elementary} $A$-{\em deformation} of $v$, and we shall write 
\[
\hspace{5truecm}v\underset{A}{\rightsquigarrow} w \quad \mbox{(or just $v\rightsquigarrow w$ when no confusion arises)},
\]
if there exists $v_1,v_2\in M$, and $a\in A$ such that $v=v_1v_2$, and $w=v_1av_2$. Then $(x,y)\in R_A$ if and only if there exists a finite sequence $z_0,\ldots,z_{k}$ of elements in $M$, with $k\geq 0$, such that
\begin{itemize}
\item[(a)] $z_0=x$, and $z_{k}=y$,
\item[(b)] for each $i\in\{0,1,\ldots,k-1\}$ either $z_i\rightsquigarrow z_{i+1}$, or $z_{i+1}\rightsquigarrow z_i$.
\end{itemize}
Since every element in $M$ is an elementary $A$-deformation of itself, this is equivalent to the existence of a finite sequence $z_0,\ldots,z_{2k}$, with $k\geq 0$, satisfying condition (a) above, and such that
\[
z_0\rightsquigarrow z_{1}\leftsquigarrow z_2\rightsquigarrow \cdots \leftsquigarrow z_{2k-2}\rightsquigarrow z_{2k-1}\leftsquigarrow z_{2k}
\]

\begin{ex}\label{congruencies_normals_a_N}{\rm
Let be $M=\NN_+$. Then $R_{\langle n\rangle}$, $n\geq 1$, is nothing but the usual congruence modulo $n$. Indeed, a positive integer $l\geq 0$ is an elementary $\langle n\rangle$-deformation of $k\geq 0$ if and only if $l\geq k$, and $l-k\in\langle n\rangle$. Moreover, if $l$ is an elementary $\langle n\rangle$-deformation of both $k$ and $k'$, and $k'\geq k$ then $k'$ is an elementary $\langle n\rangle$-deformation of $k$. Therefore $(x,y)\in R_{\langle n\rangle}$ if and only if $x,y$ differ by a multiple of $n$.
}
\end{ex}

In fact, in the commutative case, and when $A$ is a subsemigroup of $M$, as it is the case in the previous example, the congruence $R_A$ is more easily described as follows. 

\begin{prop}
Let $M$ be a commutative monoid, $A$ a subsemigroup of $M$, and $x,y\in M$. Then $(x,y)\in R_A$ if and only if there exists $a,a'\in A$ such that $x+a=y+a'$.
\end{prop}

\begin{proof}
If $M$ is commutative, we clearly have that $v\in M$ is an elementary $A$-deformation of $u\in M$ if and only if $v=u+a$ for some $a\in A$. Hence if $x+a=y+a'$ for some $a,a'\in A$ then we have $x\rightsquigarrow x+a=y+a'\leftsquigarrow y$, and $(x,y)\in R_A$. Conversely, let be $(x,y)\in R_A$. Then there exists a sequence $z_0,\ldots,z_{2k}$ as before. If $k=1$ we have $x+a=y+a'$ for some $a,a'\in R$. The cases $k\geq 2$ follow then by induction on $k$ together with the fact that $A$ is closed by sums. Thus let us assume that $z_0+a_0=z_{2k-2}+a_1$ for some $a_0,a_1$. Then from the case $k=1$ it follows that we also have $z_{2k-2}+a'_1=z_{2k}+a_2$ for some $a'_1,a_2\in A$. Hence 
\[
z_0+a_0+a'_1=z_{2k-2}+a_1+a'_1=z_{2k}+a_1+a_2
\]
with $a_0+a'_1,a_1+a_2\in A$.
\end{proof}

Clearly, if $M\neq\{1\}$ the assignments $A\mapsto R_A$ are not one-to-one. Thus for each $A\subsetneq M$ with $1\notin A$ we have $R_A=R_{A\cup\{1\}}$. In fact, $R_A$ depends on $A$ only through its normal closure so that subsets of $M$ having the same normal closure induce the same congruence. Actually, the converse is also true. To prove these claims, we need the following invariance properties of $\mathsf{Cong}(M,A)$ with respect to the subset $A$. 

\begin{prop}\label{cong(M)_diversos_subconjunts}
Let be $A$ any subset of $M$. Then:
\begin{itemize}
\item[(a)] $\mathsf{Cong}(M,A)=\mathsf{Cong}(M,\langle A\rangle')$, with $\langle A\rangle'$ the groupal submonoid of $M$ generated by $A$;
\item[(b)] $\mathsf{Cong}(M,A)=\mathsf{Cong}(M,xAy)$ for every $(x,y)\in X_A$.
\end{itemize} 
\end{prop}

\begin{proof}
Clearly, $\mathsf{Cong}(M,A)\supseteq\mathsf{Cong}(M,\langle A\rangle')$ because $A\subseteq\langle A\rangle'$. To prove the reverse inclusion let be $R\in\mathsf{Cong}(M)$ such that $A\subseteq [1]_R$. Then for every $a_1,a_2\in A$ we have $(1,a_1),(1,a_2)\in R$ and hence,
\[
(1,a_1a_2)=(1,a_1)(1,a_2)\in R.
\]
Moreover, if $a\in A$ is invertible we have $(a^{-1},a^{-1}),(1,a)\in R$ and hence,
\[
(a^{-1},1)=(a^{-1},a^{-1})(1,a)\in R.
\]
Since every element in $\langle A\rangle'$ is a finite product of elements of $A$ and/or their inverses (when they exist) it follows that $[1]_R$ also contains $\langle A\rangle'$. This proves (a). As to item (b), the inclusion $\mathsf{Cong}(M,A)\supseteq\mathsf{Cong}(M,xAy)$ is again obvious. To prove the reverse inclusion let be $R\in\mathsf{Cong}(M)$ such that $[1]_R$ contains $A$,and let be $(x,y)\in X_A$. Then there exists $a_1,a_2\in A$ such that $xa_1y=a_2$, and $(a_1,1),(1,a_2)\in R$ by hypothesis. Then for every $a\in A$ we have $(1,a)\in R$ and hence, $(a_1,a)=(a_1,1)(1,a)\in R$. Therefore 
\[
(a_2,xay)=(x,x)(a_1,a)(y,y)\in R.
\]
By transitivity, we conclude that $(1,xay)\in R$. Since this is true for each $a\in A$, and for each $(x,y)\in X_A$ it follows that $[1]_R$ also contains every generalized conjugate of $A$.
\end{proof}

\begin{prop}\label{factoritzacio_Phi'_M}
For every monoid $M$, and every subset $A$ of $M$ we have 
\[
\mathsf{Cong}(M,A)=\mathsf{Cong}(M,\mathrm{ncl}(A))
\]
and consequently, $R_A=R_{\mathrm{ncl}(A)}$. 
\end{prop}

\begin{proof}
The inclusion $\mathsf{Cong}(M,\mathrm{ncl}(A))\subseteq\mathsf{Cong}(M,A)$ is obvious. To prove the reverse inclusion, let us assume that $R\in\mathsf{Cong}(M,A)$. Then an easy induction on $n\geq 0$ using Lemma~\ref{cong(M)_diversos_subconjunts} shows that $R\in\mathsf{Cong}(M,A_n)$ for each $n\geq 0$, where $A_n$ is the sequence of groupal submonoids of $M$ defined by (\ref{def_An-1})-(\ref{def_An-2}). Since $\mathrm{ncl}(A)$ is the union of all $A_n$'s (cf. Proposition~\ref{descripcio_clausura_normal} above) it follows that $R\in\mathsf{Cong}(M,\mathrm{ncl}(A))$. The last assertion follows from the fact that $R_A$ is the minimum of the sublattice $\mathsf{Cong}(M,A)$, and $R_{\mathrm{ncl}(A)}$ is the minimum of $\mathsf{Cong}(M,\mathrm{ncl}(A))$.
\end{proof}

\subsection{Embedding the set of normal submonoids into the set of congruences}
Proposition~\ref{factoritzacio_Phi'_M} implies that subsets of $M$ having the same normal closure induce the same congruence. Actually, the converse is also true. In fact, we claim that the map $\Phi_M:\mathsf{NorSub}(M)\to\mathsf{Cong}(M)$ sending each normal submonoid $S\lhd M$ to the congruence $R_S$ induced by $S$ is one-to-one, from which the above converse readily follows. To prove this, we first need the following analog to a well known fact about the congruences on a group.

\begin{prop}\label{lema_classe_neutre}
For every congruence $R$ on a monoid $M$ the identity element equivalence class $[1]_R$ is a normal submonoid of $M$.
\end{prop}

\begin{proof}
Clearly, $[1]_R$ is a groupal submonoid of $M$ because of the compatibility of $R$. To see that it is invariant, let be $(x,y)\in X_{[1]_R}$. This means that there exist $z,z'\in[1]_R$ such that $xzy=z'$. Then by the compatibility of $R$ we have
\[
[1]_R=[z']_R=[xzy]_R=[xy]_R,
\]
i.e. $xy\in[1]_R$. Therefore for each $z_1\in[1]_R$ we have $[xz_1y]_R=[xy]_R=[1]_R$, i.e. $x[1]_Ry\subseteq [1]_R$.
\end{proof}

Then let $\Psi_M:\mathsf{Cong}(M)\to\mathsf{NorSub}(M)$ be the map given by $R\mapsto [1]_R$. If $M$ is a group $G$ we know that $\Phi_G$ is a lattice isomorphism with inverse $\Psi_G$. Although this is no longer true for arbitrary monoids, the following weaker facts still remain true in the general setting.

\begin{thm}\label{Phi_M_injectiva}
Let $M$ be a monoid. Then:
\begin{itemize}
\item[(a)] $\Phi_M$ (resp. $\Psi_M$) is a join-homomorphism (resp. a meet-homomorphism);
\item[(b)] $\Phi_M$ is a (set-theoretic) section of $\Psi_M$. In particular, it is one-to-one.
\end{itemize}
\end{thm}

\begin{proof}
The map $\Phi_M$ is order-preserving. Thus if $S\subseteq S'$ then $R_{S'}$ is a congruence containing $\{1\}\times S'$ and consequently, also $\{1\}\times S$ so that $R_S\subseteq R_{S'}$ by definition of $R_S$. Since $S,S'\subseteq S\vee S'$ it follows that $R_S,R_{S'}\subseteq R_{S\vee S'}$ and hence, $R_S\vee R_{S'}\subseteq R_{S\vee S'}$ by definition of the join $R_S\vee R_{S'}$. To prove the reverse inclusion it is enough to observe that for every normal submonoids $S,S'$ the equivalence class of the identity element modulo $R_S\vee R_{S'}$ is a normal submonoid that contains both $S$ and $S'$ and hence, also $S\vee S'$. Therefore $R_{S\vee S'}\subseteq R_S\vee R_{S'}$ by definition of $R_{S\vee S'}$. This proves that $\Phi_M$ is join-preserving. As to $\Psi_M$, it is clearly order-preserving, and for every congruences $R,R'$, and every $x\in M$ we have $[x]_{R\cap R'}=[x]_R\cap[x]_{R'}$. In particular, this is true when $x=1$, i.e. $\Psi_M$ is meet-preserving. This proves (a).

To prove (b) we have to see that $[1]_{R_S}=S$ for every normal submonoid $S$ of $M$. The inclusion $S\subseteq\,[1]_{R_S}$ follows from the definition of $R_S$. To prove the reverse inclusion, let us consider an element $z\in[1]_{R_S}$. By the previous description of $R_S$ this means that there exists a finite sequence $x_0,\ldots,x_k$ of elements in $M$ such that $x_0=1$, $x_k=z$, and either $x_i\underset{S}{\rightsquigarrow} x_{i+1}$, or $x_{i+1}\underset{S}{\rightsquigarrow} x_i$ for each $i\in\{0,1,\ldots,k-1\}$. Then to see that $z\in S$ we shall prove the following facts:
\begin{itemize}
\item[(i)] every elementary $S$-deformation of an element in $S$ is again an element in $S$, and
\item[(ii)] an element in $S$ can be an elementary $S$-deformation of an element $v\in M$ only if $v\in S$.
\end{itemize}
Clearly, if both of these facts are true every element in the above sequence $x_0,\ldots,x_k$ is in $S$ because $x_0=1\in S$. In particular, we have $z=x_k\in S$ and hence, $[1]_{R_S}\subseteq S$. Let us prove (i)-(ii).

\smallskip
\noindent
\underline{Proof of (i)}. Let be $w\in M$ a $S$-deformation of some element $s\in S$. This means that there exists $x,y\in M$, and $s'\in S$ such that $s=xy$, and $w=xs'y$. But $s=xy\in S$ implies that $w=xs'y\in xSy\subseteq S$ because of the invariance of $S$.

\smallskip
\noindent
\underline{Proof of (ii)}. Let be $s\in S$ an elementary $S$-deformation of $v\in M$. This means that there exist $v_1,v_2\in M$ and $s'\in S$ such that $v=v_1v_2$, and $s=v_1s'v_2$. In particular, we have $s\in(v_1Sv_2)\cap S$ and hence, $v=v_1v_2\in S$ because of the invariance of $S$.
\end{proof}

\begin{cor}\label{igualtat_R_T}
Let $A,B$ be subsets of $M$. Then:
\begin{itemize}
\item[(a)] $R_A=R_{B}$ if and only if $\mathrm{ncl}_M(A)=\mathrm{ncl}_M(B)$;
\item[(b)] $[1]_{R_A}=\mathrm{ncl}_M(A)$.
\end{itemize}
\end{cor}

\begin{proof}
Both items are immediate consequences of Proposition~\ref{factoritzacio_Phi'_M} and Theorem~\ref{Phi_M_injectiva}.
\end{proof}

Next corollary mimicks the fact that every normal subgroup of a group $G$ is the kernel of some group homomorphism with domain $G$, thus providing more arguments in favor of the idea that normal submonoids are the right analog for arbitrary monoids of the normal subgroups of a group.

\begin{cor}\label{submonoide_normal=nucli}
Let $M$ be a monoid. Then every normal submonoid $S$ of $M$ is the preimage $f^{-1}(1)$ of a monoid homomorphism $f:M\to N$ for some monoid $N$. 
\end{cor}

\begin{proof}
It follows from Theorem~\ref{Phi_M_injectiva} that $S$ is the preimage of the identity element by the projection of $M$ onto its quotient $M/R_S$.
\end{proof}

With additional assumptions on $M$ it is possible to go further. For instance, if $M$ is a cancellative commutative monoid satisfying condition (*) of Theorem~\ref{monoides_modulars} then we have the following stronger result.

\begin{prop}\label{Phi_M_mono}
Let $M$ be a cancellative commutative monoid satisfying (*) in Theorem~\ref{monoides_modulars}. Then $\mathsf{NorSub}(M)$ embeds isomorphically onto a sublattice of $\mathsf{Cong}(M)$. 
\end{prop}
 
\begin{proof}
It is enough to see that $\Phi_M$ is also meet-preserving when $M$ is as in the statement. For every normal submonoids $S,S'$ of any monoid $M$ we always have $R_{S\cap S'}\subseteq R_S\cap R_{S'}$ because $R_S\cap R_{S'}$ is a congruence containig $\{1\}\times(S\cap S')$. Conversely, let us assume that $(x,y)\in R_S\cap R_{S'}$. Since $M$ is commutative, this means that there exist $s_1,s_2\in S$, and $s'_1,s'_2\in S'$ such that $x+s_1=y+s_2$, and $x+s'_1=y+s'_2$. Hence
\[
x+s_1+s'_2=y+s_2+s'_2=x+s'_1+s_2
\]
and consequently, $s_1+s'_2=s'_1+s_2$ because $M$ is cancellative. Let us now observe that if $M$ satisfies (*) then every normal submonoid of $M$ also satisfies (*). Indeed, if $s_1,s_2\in S$, with $S\lhd M$, are such that $s_2=s_1+z$ for some $z\in M$ then $z\in LX_S$ and hence, $z=z+0\in S$. Then we may assume that $s'_2=s'_1+s'$ for some $s'\in S'$. Then we have
\[
s'_1+s_2=s_1+s'_2=s_1+s'_1+s'
\]
and hence, $s_2=s_1+s'$, i.e. $s'\in LX_S$. Since $S$ is invariant, it follows that $s'\in S$ and consequently, $s'\in S\cap S'$. Coming back to the initial hypothesis that $x+s_1=y+s_2$ we conclude that
\[
x+s_1=y+s_1+s'
\]
and hence $x=y+s'$, i.e. $(x,y)\in R_{S\cap S'}$.
\end{proof}

\begin{rem}\label{no_preservacio_meets}{\rm
$\Phi_M$ seems to be not meet-preserving in general. Although we always have $R_{S\cap S'}\subseteq R_S\cap R_{S'}$ it seems that the reverse inclusion can be false. In fact, $(x,y)\in R_S\cap R_{S'}$ if and only if $x,y$ can be connected by both a finite sequence of elementary $S$-deformations, and a finite sequence of elementary $S'$-deformations. However, it is not clear that this is enough for the existence of a finite sequence of elementary $(S\cap S')$-deformations connecting them. Similarly, $\Psi_M$ seems to be not join-preserving in general. We always have $[1]_R\vee[1]_{R'}\subseteq [1]_{R\vee R'}$ because $[1]_{R\vee R'}$ is a normal submonoid containing both $[1]_R$ and $[1]_{R'}$. But in general the converse $[1]_{R\vee R'}\subseteq [1]_{R}\vee [1]_{R'}$ seems to be false. }
\end{rem}

\subsection{Normal, and exceptional congruences}
Although $\Phi_M$ is always injective, it is not surjective for a generic monoid. As shown by the examples mentioned in the introduction, for a given normal submonoid $S\lhd M$ the congruence $R_S$ may be just one of several congruences on $M$ having $S$ as the equivalence class of the identity element. This suggests distinguishing the following two types of congruences on, and quotients of, a generic monoid.

\begin{defn}
A congruence $R$ on a monoid $M$ is called {\em normal} if $R=R_S$ for some normal submonoid $S$ of $M$ (equivalently, if $R$ in in the image of $\Phi_M$). Otherwise, it is called {\em exceptional}. Similarly, a quotient of $M$ is called {\em normal} (resp. {\em exceptional}) when it is the quotient of $M$ modulo a normal (resp. exceptional) congruence.
\end{defn}

\begin{ex}{\rm
It follows from Proposition~\ref{submonoides_invariants_de_N} that the non-trivial normal congruences in the additive monoid $\NN_+$ are the congruences $R_{0,n}=\{(i,j)\in\NN\times\NN\,|\,i\equiv j\ \mbox{(mod. $n$)}\}$ for each $n\geq 1$, with $[0]=\langle n\rangle$. Up to isomorphism, the corresponding normal quotients are the cyclic groups $(\ZZ_n,+,0)$.
}
\end{ex}

Let us denote by $\mathsf{NorCong}(M)$ the subset of $\mathsf{Cong}(M)$ consisting of the normal congruences on $M$ ordered by inclusion. Then Theorem~\ref{Phi_M_injectiva} can be restated as follows. 

\begin{cor}\label{NorCon(M)}
For every monoid $M$, $\mathsf{NorCon}(M)$ is a complete lattice isomorphic to $\mathsf{NorSub}(M)$.
\end{cor}

\begin{proof}
It follows from Theorem~\ref{Phi_M_injectiva} that the corestriction of $\Phi_M$ to $\mathsf{NorCon}(M)$ is an order-preserving bijection from $\mathsf{NorSub}(M)$ to $\mathsf{NorCon}(M)$ whose inverse map is given by $R\mapsto [1]_R$ and hence, order-preserving.
\end{proof}

Let us emphasize that $\mathsf{NorCon}(M)$ seems to be just a join subsemilattice of $\mathsf{Cong}(M)$ because the meet of two generic normal congruences $R_S,R_{S'}$ in $\mathsf{NorCong}(M)$ is $R_{S\cap S'}$ but not necessarily in $\mathsf{Cong}(M)$ (cf. Remark~\ref{no_preservacio_meets} above).

\subsection{The set of congruences as a ``blow up'' of the set of normal submonoids}\label{blow_up}
For each normal submonoid $S$ of $M$ let be
\[
\mathsf{Cong}_S(M):=\{R\in\mathsf{Cong}(M)\,|\,[1]_R=S\}.
\]
Notice that $\mathsf{Cong}_S(M)\subseteq\mathsf{Cong}(M,S)$, and that it contains the minimal element $R_S$, and unique normal congruence in  $\mathsf{Cong}(M,S)$. When $S=\{1\}$ we shall write $\mathsf{Cong}_1(M)$ (or $\mathsf{Cong}_0(M)$ in the additive case) instead of $\mathsf{Cong}_{\{1\}}(M)$. Its elements are the {\em unital congruences} on $M$ mentioned in the introduction.

When ordered by inclusion, $\mathsf{Cong}_S(M)$ is a meet subsemilattice of $\mathsf{Cong}(M,S)$ and hence, of $\mathsf{Cong}(M)$. However, it is not a join subsemilattice in general. Thus if $R,R'\in\mathsf{Cong}_1(M)$ then their join $R\vee R'$ in $\mathsf{Cong}(M,S)$ (and in $\mathsf{Cong}(M)$) need not be a congruence in $\mathsf{Cong}_S(M)$ because we only have the inclusion $[1]_R\vee[1]_{R'}\subseteq[1]_{R\vee R'}$.

It readily follows from Theorem~\ref{Phi_M_injectiva} that, as a set, $\mathsf{Cong}(M)$ is in bijection with the ``blow up'' of $\mathsf{NorSub}(M)$ (equivalently, of $\mathsf{NorCong}(M)$; cf. Corollary~\ref{NorCon(M)}) obtained when each normal submonoid $S$ (normal congruence $R_S$) is replaced by the set $\mathsf{Cong}_S(M)$. It turns out that, up to isomorphism, $\mathsf{Cong}_S(M)$ is given by the set of unital congruences on the quotient monoid $M/R_S$, or $M/S$ for short. More precisely, the following holds true.

\begin{prop}\label{iso_congruencies_unitaries_a_M/S}
Let be $M$ a monoid, and $S\lhd M$. Then $\mathsf{Cong}_S(M)$ is isomorphic, as a meet semilattice, to $\mathsf{Cong}_1(M/S)$. In particular, $R_S$ is mapped through the isomorphism to the trivial unital congruence on $M/S$.
\end{prop}

\begin{proof}
Let us first observe that every congruence $R\in\mathsf{Cong}_S(M)$ induces a unital congruence on the quotient $M/S$ given by
\[
\tilde{R}:=\{([x]_{R_S},[y]_{R_S})\in M/S\times M/S\,|\, (x,y)\in R\}.
\]
Indeed, if $(x,y)\in R$, and $x',y'\in M$ are such that $(x',x)\in R_S$, $(y,y')\in R_S$ then we also have $(x,x')\in R$, and $(y,y')\in R$ because $R_S$ is the minimum of $\mathsf{Cong}_S(M)$. By transitivity, it follows that $(x',y')\in R$ and hence, $\tilde{R}$ is a well defined binary relation on $M/S$. Moreover, it is immediate to check that $\tilde{R}$ is a congruence because this is true for $R$. Finally, $[x]_{R_S}\tilde{R}\,[1]_{R_S}$ means that $x\,R\,1$ and hence, $x\in[1]_{R}=S=[1]_{R_S}$, i.e. $\tilde{R}$ is unital. In particular, when $R$ is $R_S$ then $\tilde{R}$ is the trivial congruence on $M/S$. Conversely, every unital congruence $T$ on $M/S$ induces a congruence $T^*$ on $M$ given by
\[
T^*:=\{(x,y)\in M\times M\,|\,([x]_{R_S},[y]_{R_S})\in T\},
\]
and $T^*$ is such that $[1]_{T^*}=S$. Indeed, $[[1]_{R_S}]_T=\{[1]_{R_S}\}$ because $T$ is unital. Hence for each $x\in M$ we have
\[
(1,x)\in T^*\ \Leftrightarrow\ ([1]_{R_S},[x]_{R_S})\in T \ \Leftrightarrow [x]_{R_S}=[1]_{R_S}\ \Leftrightarrow\ (1,x)\in R_S\ \Leftrightarrow\ x\in S.
\]
Then it is easy to check that the maps $t_S:\mathsf{Cong}_S(M)\to\mathsf{Cong}_1(M/S)$, and $t^*_S:\mathsf{Cong}_1(M/S)\to\mathsf{Cong}_S(M)$ respectively given by $R\mapsto\tilde{R}$, and $T\mapsto T^*$ are inverse of each other. Moreover, both are order-preserving and hence, they define a meet semilattice isomorphism.
\end{proof}

This result, together with Theorem~\ref{Phi_M_injectiva}, leads to the following description of the set of congruences on an arbitrary monoid.

\begin{thm}\label{conjunt_congruencies_M}
For every monoid $M$ there is a (set-theoretic) bijection
\begin{equation}\label{formula_Cong(M)}
\mathsf{Cong}(M)\longrightarrow\bigcup_{S\in\mathsf{NorSub}(M)}\mathsf{Cong}_1(M/S)
\end{equation}
given by $R\mapsto t_{[1]_R}(R)$, and whose inverse is given by $T\mapsto t^*_S(T)$ for each $T\in\mathsf{Cong}_1(M/S)$. Moreover, the exceptional congruences on $M$ are mapped through this bijection to the non-trivial unital congruences on the normal quotients of $M$. 
\end{thm}

\noindent
This theorem reduces the computation of the whole set of congruences on an arbitrary monoid $M$ to the computation of:
\begin{itemize}
\item[(1)] its set $\mathsf{NorSub}(M)$ of normal submonoids (equivalently, normal congruences), 
\item[(2)] the normal quotient $M/S$ for each normal submonoid $S$ of $M$, and
\item[(3)] the set $\mathsf{Cong}_1(M/S)$ of unital congruences on each normal quotient $M/S$.
\end{itemize}
As it becomes clear from the bijection (\ref{formula_Cong(M)}), the higher complexity of the theory of congruences on generic monoids with respect to the corresponding theory for groups ultimately comes from the existence of non-trivial unital congruences on a monoid. For every group $G$ there is only one unital congruence on $G$, so that every term in the right hand side of (\ref{formula_Cong(M)}) is a singleton. In fact, the following characterizations of the unital congruences on a generic monoid are straightforward.

\begin{prop}
Let be $R$ a congruence on a monoid $M$. Then the following are equivalent:
\begin{itemize}
\item[(i)] $R$ is unital;
\item[(ii)] for each $x,y\in M$ if $(x,y)\in R$, and $x\in U(M)$ then $x=y$;
\item[(iii)] the restriction of $R$ to the group of units $U(M)$ is the identity relation $\Delta_{U(M)}$.
\end{itemize}
\end{prop}

\begin{proof}
(i) $\Rightarrow$ (ii). Let $R$ be unital, so that $[1]_R=\{1\}$, and let be $x,y\in M$ such $x\in U(M)$, and $(x,y)\in R$. Then $(1,x^{-1}y),(1,yx^{-1})\in R$ and hence, $x^{-1}y=yx^{-1}=1$, i.e. $x=y$. Let us denote by $R_{\downarrow U(M)}$ the restriction of $R$ to $U(M)$. Clearly, $R$ is unital if $R_{\downarrow U(M)}=\Delta_{U(M)}$., and 
\end{proof}

\subsection{Congruentially simple monoids}
As poined out before, the exceptional congruences on a monoid $M$ are in bijection with the non-trivial unital congruences on its normal quotients. This suggests introducing the following definition.

\begin{defn}
A monoid $M$ is called {\em congruentially simple} if its only normal quotient having non-trivial unital congruences is the monoid $M$ as normal quotient of itself.
\end{defn}

For such monoids, the equivalence class of the identity element is trivial in all exceptional congruences, and for every normal submonoid $S\neq\{1\}$ there is a unique congruence $R$ such that $[1]_R=S$. In other words, a monoid is congruentially simple if all its congruences are uniquely determined by the equivalence class of the identity element except when this class is trivial. Once more, the question arises whether every monoid is congruentially simple, and in case it is no, to identify families of monoids which are congruentially simple.

\begin{ex}{\rm
$\NN_+$ is congruentially simple. Indeed, it follows from Proposition~\ref{submonoides_invariants_de_N} and Example~\ref{congruencies_normals_a_N} that, up to isomorphism, the only non-trivial, proper normal quotients of $\NN_+$ are the additive groups $\ZZ_n$ for each $n\geq 2$. Being all of them groups, there are no non-trivial unital congruences on them. Hence $\mathsf{Cong}(\NN_+)$ consists of the normal congruences $\Delta_\NN$ and $R_{0,n}$ for each $n\geq 1$, together with the non-trivial unital congruences on $\NN_+$.
}
\end{ex}

As discussed in the next subsection, the finite full transformation monoids $T_n$ are also congruentially simple.

\subsection{Malcev's theorem revisited}
In a now classical paper from 1952, Malcev computed the lattice of congruences of the finite full transformation monoids $T_n$, $n\geq 1$ \cite{Malcev1952}. The computation starts with the fact that the lattice of ideals of $T_n$ is given by the chain
\[
I_1\subseteq I_2\subseteq\cdots\subseteq I_n=T_n,
\]
where $I_k$ stands for the set of endomorphisms of $\nbf:=\{1,\ldots,n\}$ of rank $\leq k$ (by the rank of such an endomorphism it is meant the cardinal of its image). Malcev realized that for every congruence $R$ on $T_n$ different from the uniform congruence $\nabla_{T_n}:=T_n\times T_n$ there is some $k\in\{1,\ldots,n\}$, and some normal subgrup $N\lhd S_k$ such that $R$ is the identity on $I_n\setminus I_k$, it identifies all endomorphisms in $I_{k-1}$, and its restriction to $I_k\setminus I_{k-1}$ is completely given by $N$. To be precise, two endomorphisms $u,v\in I_k\setminus I_{k-1}$ are equivalent if and only if $u=v$ or the following three conditions hold:
\begin{itemize}
\item[(1)] both have the same image $\{i_1,\ldots,i_k\}\subseteq\nbf$, 
\item[(2)] $u^{-1}(i_j)=v^{-1}(i_j)$ for each $j=1,\ldots,k$, and 
\item[(3)] there exists a permutation $\tau\in N$ such that $v=\tau u$.
\end{itemize}
Let be $R_{k,N}$ the congruence so defined by the pair $(k,N)$. For instance, $R_{1,S_1}$ is the identity congruence $\Delta_{T_n}$, and $R_{k,\{1\}}$ for any $k\geq 2$ is the Rees congruence $R_{I_{k-1}}$ mentioned in the introduction. Clearly, if $k\leq k'$ then $R_{k,N}\subseteq R_{k',N'}$ for every normal subgroups $N\lhd S_k$, and $N'\lhd S_{k'}$ with $N\subseteq N'$ when $k=k'$. It follows that the whole lattice of congruences on $T_n$ is given by the chain
\begin{align*}
&\mbox{if $n=2$:}\ \ \Delta_{T_2}\subseteq R_{I_1}\subseteq R_{2,S_2}\subseteq\nabla_{T_2} ,
\\
&\mbox{if $n=3$:}\ \ \Delta_{T_3}\subseteq R_{I_1}\subseteq R_{2,S_2}\subseteq R_{I_2}\subseteq R_{3,A_3}\subseteq R_{3,S_3}\subseteq\nabla_{T_3},
\\
&\mbox{if $n=4$:}\ \ \Delta_{T_4}\subseteq R_{I_1}\subseteq R_{2,S_2}\subseteq R_{I_2}\subseteq R_{3,A_3}\subseteq R_{3,S_3}\subseteq  R_{I_3}\subseteq R_{4,K_4}\subseteq R_{4,A_4}\subseteq R_{4,S_4}\subseteq\nabla_{T_4},
\end{align*}
and by the chain
\begin{align}
\Delta_{T_n}\subseteq R_{I_1}\subseteq R_{2,S_2}&\subseteq R_{I_2}\subseteq R_{3,A_3}\subseteq R_{3,S_3}\subseteq  R_{I_3} \nonumber
\\
&\subseteq R_{4,K_4}\subseteq R_{4,A_4}\subseteq R_{4,S_4}\subseteq
\cdots\subseteq R_{I_{n-1}}\subseteq R_{n,A_n}\subseteq R_{n,S_n}\subseteq\nabla_{T_n},\label{reticle_congruencies_Tn}
\end{align}
for each $n\geq 5$, where $K_4$ denotes the Klein four-group; see \cite{Ganyushkin-Mazorchuk2009} for a modern presentation of the subject, including the computation of the lattice of congruences of the related monoids $PT_n$ of partial transformations, and $IS_n$ of partial injective transformations. Notice that the equivalence class of the identity element in $R_{k,N}$ is
\[
[1]_{R_{k,N}}=\left\{
\begin{array}{ll}
\{1\}, & \mbox{if $k<n$,}
\\
N, & \mbox{if $k=n$.}
\end{array}\right.
\]
Hence the normal congruences on $T_n$ are the congruences $\Delta_{T_n}$, $R_{n,A_n}$, $R_{n,S_n}$, and $\nabla_{T_n}$ (and $R_{4,K_4}$ when $n=4$), and all exceptional congruences are in this case unital, i.e. $T_n$ is congruentially simple. 

Ulterior works devoted to the computation of the lattice of congruences on other specific families of monoids often follow Malcev strategy, starting with the lattice of ideals of the monoid or, more precisely, with some ascending chain of ideals whose union is the whole monoid; see, in particular, the works by East and Ruskuc (\cite{East-Ruskucd},\cite{East-Ruskuc2022b},\cite{East-Ruskuc2022a}, \cite{East-Ruskuc2022c}). However, it is also possible to address the problem using Theorem~\ref{conjunt_congruencies_M}, whose starting point is the lattice of normal submonoids. As mentioned above, and although at some point we may be forced to converge to a strategy similar to Malcev's, in particular, in the computation of the lattice of unital congruences of the normal quotients, this theorem offers a new strategy, and this new strategy may allow us to achieve some results in a different way. This is the case of the congruential simplicity of $T_n$, which can be proved directly without using Malcev result.

Indeed, let us consider the generic case $n\geq 5$. We know from Theorem~\ref{conjunt_congruencies_M} that the exceptional congruences on $T_n$ are in bijection with the non-trivial unital congruences on the normal quotients of $T_n$. Hence a path to identify them consists of computing first all normal quotients, and then looking for the unitary congruences on each of these quotients. 

\begin{prop}\label{quocients_normals_Tn}
Up to isomorphism, the non-trivial, proper normal quotients of $T_n$ for $n\geq 5$ are the multiplicative monoids $\{0,1\}$ and $\{-1,0,1\}$.
\end{prop}

\begin{proof}
It follows from Example~\ref{reticle_submonoides_normals_Tn} that the only non-trivial, proper normal quotients of $T_n$ for $n\geq 5$ are $T_n/S_n$, and $T_n/A_n$. Hence it is enough to prove that:
\begin{itemize}
\item[(a)] $T_n/S_n\cong(\{0,1\},\cdot,1)$;
\item[(b)] $T_n/A_n\cong(\{-1,0,1\},\cdot,1)$.
\end{itemize}

\medskip
\noindent
\underline{Proof of (a)}. Recall that $T_n/S_n$ means the quotient of $T_n$ modulo the smallest congruence $R=R_{S_n}$ on $T_n$ such that $[1]_{R}=S_n$. In particular, for each $u\in S_n$ we have $[u]_R=S_n$. It turns out that $[u]_R=T_n\setminus S_n$ if $u\notin S_n$. To be precise, if $c_1:\nbf\to\nbf$ denotes the constant map defined by $c_1(k)=1$ for each $k\in\nbf$ let us prove that
\begin{equation}\label{classe_c1}
(u,c_1)\in R \ \Leftrightarrow\ u\notin S_n.
\end{equation}
The implication to the right is an immediate consequence of the fact that $[1]_R=S_n$. Thus if $(u,c_1)\in R$ for some $u\in S_n$ then we also have $(1,c_1)\in R$ and hence, $c_1\in[1]_R=S_n$ which is clearly false. To prove the converse, we proceed by induction on the rank $r\in\{1,\ldots,n-1\}$ of $u$ (i.e. on the cardinal of its image). If $r=1$ then $u$ is the constant map $c_i$ sending every $k\in\nbf$ to $i$ for some $i\in\nbf$, and $c_1$ is clearly an elementary $S_n$-deformation of $c_i$. Indeed, $c_1$ is the composite of $c_i$ with the transposition $(1i)\in S_n$. Hence $(u,c_1)\in R$. Let us now assume that for some $r\in\{1,\ldots,n-2\}$ every map of rank $r$ is $R$-related to $c_1$, and let be $u$ of rank $r+1$. Since $r<n-1$ we still have $r+1<n$ and hence, there exists some $i\in\nbf$ not in the image of $u$. Let be $j$ any element in the image of $u$, and $l$ such that $u(l)\neq j$. Such an $l$ exists because $u$ is of rank $r+1\geq 2$. Then let us consider the map $u':\nbf\to\nbf$ defined by
\[
u'(k)=\left\{\begin{array}{ll}
k, & \mbox{if $k\neq i$,}
\\
j, & \mbox{if $k=i$.}
\end{array}\right.
\]
Clearly, we have $u'u=u$. Hence the composite $\hat{u}=u'\tau u$, with $\tau$ the transposition $(iu(l))\in S_n$, is an elementary $S_n$-deformation of $u$ and consequently, $(u,\hat{u})\in R$. But $\hat{u}$ is of rank $r$ because all elements in $u^{-1}(u(l))$ are now mapped not to $u(l)\neq j$ but to $j$, while $\hat{u}(k)=u(k)$ for each $k\notin u^{-1}(u(l))$. By the induction hypothesis, $(\hat{u},c_1)\in R$ and hence, $(u,c_1)\in R$ by transitivity. This proves (\ref{classe_c1}).

It follows that the map $T_n\to\{0,1\}$ sending each $u\notin S_n$ to $0$, and each $u\in S_n$ to $1$ is a monoid epimorphism that factors through the quotient $T_n/S_n$, and the induced map $T_n/S_n\to\{0,1\}$ is a monoid isomorphism.

\medskip
\noindent
\underline{Proof of (b)}. $T_n/A_n$ means the quotient of $T_n$ modulo the smallest congruence $R=R_{A_n}$ on $T_n$ such that $[1]_{R}=A_n$. In particular, for each $u\in A_n$ we have $[u]_R=A_n$. It turns out now that $[u]_R=S_n\setminus A_n$ if $u\in S_n\setminus A_n$, and $[u]_R=T_n\setminus S_n$ if $u\notin S_n$. To be precise, let us prove that
\begin{align}
(u,c_1)\in R \ &\Leftrightarrow\ u\notin S_n,
\label{classe_c1_bis}
\\
(u,(12))\in R \ &\Leftrightarrow\ u\in S_n\setminus A_n.
\label{classe_(12)}
\end{align}
The implication to the right  in (\ref{classe_c1_bis}) follows from the fact that $[u]_R\subseteq S_n$ for each $u\in S_n$. To see this, recall that $(u,v)\in R$ if and only if there exists a finite sequence $w_0,\ldots,w_k$ of elements in $T_n$ such that $u=w_0$, $v=w_k$, and for each $i\in\{0,1,\ldots,k-1\}$ either $w_i$ or $w_{i+1}$ is an elementary $A_n$-deformation of the other. Actually, since $A_n$ is a group both conditions are equivalent. For instance, if $w_{i+1}=z_1\tau z_2$ for some $z_1,z_2\in T_n$ such that $w_i=z_1z_2$, and some $\tau\in A_n$ then $w_i=(z_1\tau)\tau^{-1} z_2$. Therefore, it is enough to see that every elementary $A_n$-deformation of an element in $S_n$ is still in $S_n$, and this is clearly true because every factorization of an element in $S_n$ is necessarily as a composite of elements in $S_n$. To prove the converse we proceed again by induction on the rank $r$ of $u$. If $r=1$ then $u$ is the constant map $c_i$ for some $i\in\nbf$, and when $i\neq 1$ the map $c_1$ is the composite of $c_i$ and the permutation $(1ji)\in A_n$ for any $j\neq 1,i$ (here we use that $n\geq 3$). Hence $(u,c_1)\in R$. Let us now assume that for some $r\in\{1,\ldots,n-2\}$ every map of rank $r$ is $R$-related to $c_1$, and let be $u$ of rank $r+1$. Let be $i$ not in the image of $u$, $j$ any element in the image of $u$, and $l$ such that $u(l)\neq j$ (we are using again that $n\geq 3$), and let be $u':\nbf\to\nbf$ the same map as in the proof of the previous proposition. In particular, we have $u'u=u$. Then the composite $\tilde{u}=u'\sigma u$ with $\sigma$ the permutation $(iu(l)j)\in A_n$ is an elementary $A_n$-deformation of $u$ and consequently, $(u,\tilde{u})\in R$. But $\tilde{u}$ is of rank $r$ because all elements in $u^{-1}(u(l))$ are now mapped to $j$, while $\tilde{u}(k)=u(k)$ for each $k\notin u^{-1}(u(l))$. Indeed, $\sigma u$ maps all of $u^{-1}(j)$ to $i$ but this is next mapped again to $j$ by $u'$. By the induction hypothesis, $(\tilde{u},c_1)\in R$ and hence, $(u,c_1)\in R$. This proves (\ref{classe_c1_bis}).

Let us now prove (\ref{classe_(12)}).  The implication to the right follows now from the fact that for every odd permutation (for instance, $(12)$) $[\sigma]_R\subseteq S_n\setminus A_n$. Indeed, every factorization of an odd permutation is necessarily as a composite of an odd, and an even permutation and hence, every elementary $A_n$-deformation of it is also an odd permutation. The converse follows from the obvious fact that every odd permutation $u$ can be connected to $(12)$ by a finite number of elementary $A_n$-deformations. It is enough to write $u$ as the composite of an odd number of transpositions including $(12)$ (two consecutive times if necessary), and then proceed to eliminate consecutive pairs of transpositions by performing elementary $A_n$-deformations until we are left only with the transposition $(12)$. This proves (\ref{classe_(12)}).

By an easy case-by-case check it follows that the map $T_n\to\{-1,0,1\}$ sending each $u\notin S_n$ to $0$, each $u\in A_n$ to $1$, and each $u\in S_n\setminus A_n$ to $-1$ preserves products and hence, is a monoid epimorphism that factors through the quotient $T_n/A_n$, and whose induced map $T_n/A_n\to\{-1,0,1\}$ is a monoid isomorphism.
\end{proof}

\begin{cor}
$T_n$ is congruentially simple for each $n\geq 5$.
\end{cor}

\begin{proof}
Neither of the monoids in Proposition~\ref{quocients_normals_Tn} have unital congruences other than the equality. 
\end{proof}

\section{Final comments}

There remain many open questions in this work which are left for future work. Let us mention a few of them.

\begin{itemize}
\item[(1)] There is first the question of the modularity of the lattice of normal submonoids of a monoid. Although it is unlikely that every monoid is modular, we do have no example of monoid whose lattice is not modular. Assuming that example exists, it will be interesting to identify where the modularity of a monoid is hidden, namely, a necessary and sufficient condition for a monoid to be modular.
\item[(2)] There is also the question of the normality of a monoid, and the associated question of the normal simplicity when the monoid is normal. Although it seems possible that the group of units of a monoid is always a normal submonoid of it, we have not been able to prove it. In any case, it is clear that not every normal monoid is normally simple, as the example of the additive monoid of natural numbers shows. Therefore it will be also interesting to identify necessary and/or sufficient conditions for a normal monoid to be normally simple. 
\item[(3)] We have seen that congruences on a monoid can be naturally classified into normal congruences, and exceptional ones. Congruences of the last type are those for which there exists a strictly smaller congruence having the same normal submonoid as equivalence class of the identity element. The question arises whether exceptional congruences only exist for the trivial submonoid, as it is the case of both $\NN_+$, and the finite full transformations monoids, or they also exist for other normal submonoids. This is the question of determining whether every monoid is congruentially simple or not. If not, the problem arises of determining what normal submonoids of a non-congruentially simple monoid are what might be called ``congruentially complete'', i.e. such that there is only one congruence with this normal submonoid as equivalence class of the identity element (equivalently, such that the normal quotient has no unital congruence other than the equality).
\item[(4)] We have shown that, in general, the lattice of normal submonoids of a monoid can be identified with just a join subsemilattice of the whole lattice of congruences on it. Are the groups the only monoids for which $\mathsf{NorSub}(M)$ and $\mathsf{Cong}(M)$ are isomorphic? 
\item[(5)] We have seen a way of reducing the computation of the congruences on a monoid to being able to compute the unital ones. What input or inputs in the monoid determine a unital congruence? In other words, what additional input or inputs, together with the equivalence class of the identity element, determine completely a congruence on a monoid? Here the lattice of ideals, and the set of non-identity idempotents of the monoid perhaps play a crucial role. In fact, in the case of groups both things collapse.
\end{itemize}

\noindent
\bibliography{congruencies_versio_penjada_arxiv}{}

\begin{thebibliography}{10}

\bibitem{Cain2020}
Alan~J. Cain.
\newblock {\em Nine Chapters on the Semigroup Art}.
\newblock Version 0.66.63 (2020-06-29), available at
  http://www-groups.mcs.st-andrews.ac.uk/~alanc/pub/c\_semigroups/, 2020.

\bibitem{East-Ruskucd}
J.~East and N.~Ruskuc.
\newblock Congruence lattices of ideals in categories and (partial) semigroups.
\newblock {\em Preprint, Arxiv: 2001.01909}.

\bibitem{East-Ruskuc2022b}
J.~East and N.~Ruskuc.
\newblock Classification of congruences of twisted partition monoids.
\newblock {\em Adv. Math.}, 395(Paper No. 108097):65pp, 2022.

\bibitem{East-Ruskuc2022a}
J.~East and N.~Ruskuc.
\newblock Congruences on infinite partition and partial {B}rauer monoids.
\newblock {\em Mosc. Math. J.}, 22(2):295--372, 2022.

\bibitem{East-Ruskuc2022c}
J.~East and N.~Ruskuc.
\newblock Properties of congruences of twisted partition monoids and their
  lattices.
\newblock {\em J. Lond. Math. Soc.}, 106(1):311--357, 2022.

\bibitem{Facchini-Rodaro2017}
A.~Facchini and E.~Rodaro.
\newblock Equalizers and kernels in categories of monoids.
\newblock {\em Semigroup Forum}, 95:455--474, 2017.

\bibitem{Ganyushkin-Mazorchuk2009}
O.~Ganyushkin and V.~Mazorchuk.
\newblock {\em Classical Finite Transformation Semigroups. An Introduction}.
\newblock Algebra and Applications, vol. 9, Springer, 2009.

\bibitem{Howie1995}
John~M. Howie.
\newblock {\em Fundamentals of Semigroup Theory}.
\newblock London Mathematical Society Monographs, New Series 12, Clarendon
  Press, Oxford, 1995.

\bibitem{Jacobson1985-I}
Nathan Jacobson.
\newblock {\em Basic Algebra I}.
\newblock 2d. edition, Dover Publications, 2009.

\bibitem{Malcev1952}
A.~I. Malcev.
\newblock Symmetric groupoids.
\newblock {\em Mat. Sbornik, N.S. (Russian)}, 31(73):136--151, 1972; English
  translation in {\em Twelve papers in Logic and Algebra}, Am. Math. Soc.
  Translations, Ser 2, 113, AMS, 1979, pp. 235-250.

\bibitem{MartinsFerreira-Sobral2022}
N.~Martins-Ferreira and M.~Sobral.
\newblock On the normality of monoid monomorphisms.
\newblock {\em Preprint, arXiv: 2210.03490}, page pp. 18, 2022.

\bibitem{Steinberg2016}
Benjamin Steinberg.
\newblock {\em Representation Theory of Finite Monoids}.
\newblock Universitext, Springer, 2016.

\end{thebibliography}
\bibliographystyle{plain}

\end{document}